\documentclass[11pt,reqno]{amsart}
\usepackage{graphicx}
\pdfoutput=1
\usepackage{amsfonts,amsmath,amssymb}
\usepackage{hyperref}
\usepackage{paralist}
\usepackage[linesnumbered,ruled,vlined,norelsize]{algorithm2e}

\newcommand{\cx}{{\mathbb{C}}}

\newcommand {\Q}{\mathcal Q}
\DeclareMathOperator{\ord}{ord}
\usepackage{thmtools}
\declaretheoremstyle[bodyfont=\normalfont]{noncursive}
\declaretheorem{theorem}
\declaretheorem[numberlike=theorem]{lemma}
\declaretheorem[numberlike=theorem]{conjecture}
\declaretheorem[numberlike=theorem]{proposition}
\declaretheorem[numberlike=theorem]{corollary}
\declaretheorem[style=noncursive,numberlike=theorem]{definition}

\declaretheorem[style=noncursive,numberlike=theorem]{remark}

\declaretheorem[style=noncursive,numberlike=theorem]{problem}
\newcommand{\im}{\ensuremath{\mbox{\rm Im}\,}}
\newcommand{\re}{\ensuremath{\mbox{\rm Re}\,}}

\newcommand{\CC}[1]{\mathbb{C}^{#1}}
\newcommand{\CP}[1]{\mathbb{CP}^{#1}}
\newcommand{\RR}[1]{\mathbb{R}^{#1}}

\newcommand{\lr}{\longrightarrow}

\numberwithin{equation}{section}
\newcommand{\diffcr}[1]{\rm{Diff}_{CR}^{#1}}
\newcommand{\Hol}[1]{\rm{Hol}^{#1}}
\newcommand{\Aut}[1]{\rm{Aut}^{#1}}
\newcommand{\hol}[1]{\mathfrak{hol}^{#1}}
\newcommand{\aut}[1]{\mathfrak{aut}^{#1}}

\title[On the analyticity of CR-diffeomorphisms]{On the analyticity of CR-diffeomorphisms}
\author {I. Kossovskiy}
\address{Department of Mathematics, University of Vienna}
\email{ilya.kossovskiy@univie.ac.at}
\author {B. Lamel}
\address{Department of Mathematics, University of Vienna}
\email{bernhard.lamel@univie.ac.at}

\begin{document}

\date{\today}


\begin{abstract}
In any positive CR-dimension and CR-codimension we provide a
construction of real-analytic holomorphically nondegenerate
CR-submanifolds, which are $C^\infty$ CR-equivalent, but are
inequivalent holomorphically. As a corollary, we provide the
negative answer to the conjecture of Ebenfelt and Huang \cite{eh}
on the analyticity of CR-equivalences between real-analytic Levi
nonflat hypersurfaces in dimension 2.
\end{abstract}

\maketitle

\tableofcontents

\section{Introduction}

Study of germs of CR-mappings between real submanifolds in complex space was initiated in the classical
work of Poincare \cite{poincare} and Cartan \cite{cartan}. Starting from the results of Cartan in \cite{cartan}, establishing, in particular,
the analyticity property for smooth CR-diffeomorphisms between Levi-nondegenerate real-analytic hypersurfaces in $\CC{2}$, the problem of
regularity of CR-mappings between various classes of real submanifolds became one of the central questions in Cauchy-Riemann geometry. Because of the importance
of the problem for  Complex Analysis and  Linear PDEs, substantial work has been done (see, e.g., \cite{pinchuksib}, \cite{chern},\cite{han},\cite{bjt},
\cite{tumanov},\cite{bhr},\cite{hu1},\cite{DiPi},\cite{ebenfelt}) in order to extend Cartan\rq{}s phenomenon to more general
classes of real submanifolds. It was a long-standing problem (see, e.g., \cite{eh}) whether one can establish the analyticity property for $C^\infty$ CR-diffeomorphisms
between merely Levi nonflat real-analytic hypersurfaces.
The main result of the paper provides a construction,
giving the negative resolution to this problem. The construction employs a recent technique (see \cite{divergence,nonminimalODE}) suggesting to replace
CR-manifolds with CR-degeneracies by  appropriate holomorphic dynamical system, and then study mappings between them accordingly. We give below a short background, outline the history of the problem,  and formulate our results in detail.

Consider germs $(M,p)$, $(M',p')$  of real-analytic submanifolds
of some $\CC{N}$. The \emph{complex tangent bundle} of $M$ is given by
$T^c M = TM \cap i TM$, and we say that $M$ is a CR-manifold if the
fiber dimension of this bundle is constant. A germ of a map
  $H\colon(M,p)\to (M',p')$ is CR if $TH (T^c M) \subset T^c M'$ and
  $TH$ is complex linear on $T^c M$. Equivalently, $H$ is CR if its
  components are germs of
   CR-functions, where a CR-function is defined as a CR-map $(M,0) \to \CC{}$. It turns out that a
  function is CR if and only if it
  is annihilated by every section of $\mathcal{V} (M) $, the \emph{CR-bundle} of $M$, which is defined  by
 $$\mathcal{V} (M) = T^{(0,1)}\CC{N} \cap \CC{} TM .$$ 
 Thus CR-maps
  satisfy a certain system of PDEs, also known as the tangential
  Cauchy-Riemann equations. Restrictions or boundary values
   of holomorphic maps are the primary examples of such maps. Note that a real-analytic CR-map is always a restriction of a map, holomorphic in an open neighborhood of the source manifold.
   
   The naturally arising problem of regularity of CR-mappings is of fundamental importance for the study of boundary regularity of  holomorphic mappings (see, e.g., the discussions in \cite{forstneric},\cite{ber}). On the other hand,
the problem of analyticity of CR-mappings is equally important for Linear PDEs,
where the latter property is addressed as \it hypoellipticity \rm and can be of substantial help for studying
regularity of solutions for a wide range of PDE systems (see \cite{cardaro}).

It turns out that systems of PDEs, determining the space of CR-mappings between real submanifolds in complex space, are rather hard to satisfy. Actually,
a heuristic going back to Poincare tells us that there are no CR-maps
between two randomly chosen CR-manifolds. This lack of richness is made
up for by a number of beautiful properties CR-maps possess: in particular,
they have an uncanny tendency to be very regular. In the case of hypersurfaces in $\CC{2}$ this regularity is already apparent in E. Cartan's work on Levi-nondegenerate germs \cite{cartan}.
Actually, every formal map between such hypersurfaces is convergent,
and every smooth CR-diffeomorphism is the restriction of a germ of a holomorphic
map. Regularity results of this sort hold under less stringent conditions.
For hypersurfaces in $\CC{2}$, it has been known for some time that
if $ M$ is minimal at $p$, then every germ of a smooth  CR
diffeomorphism  (it is enough to assume just continuity)
is actually the restriction of a germ of a holomorphic map
(see Huang~\cite{hu1}). Here minimality (or, finite type, which in
the case of real-analytic hypersurfaces is the same)
refers to the fact that
the tangential CR-equations satisfy Hormander's bracket condition, or, equivalently,
that there does not exist a germ of a complex curve $X\subset M$ through
$p$.

This regularity property relies on two crucial ingredients. One uses
the minimality to obtain a one-sided extension of the map, which relies
on the one-sided extension of the component CR-functions, possible by
results of Tumanov \cite{tumanov} (in the case of $\CC{2}$, this
result goes back to Trepreau \cite{trepreau}). One then obtains
the extension across the hypersurface by reflection methods (regularity
results of this form are therefore also known as {\em reflection principles}).
The nondegeneracy
properties of real-analytic submanifolds governing reflection
are by now well understood. One of the most useful results in that regard
is the Baouendi-Jacobowitz-Treves theorem \cite{bjt} which states
that every smooth boundary value of a holomorphic map in a wedge actually extends to
a germ of a holomorphic map, if the target real submanifold is {\em essentially finite}.
The reflection principle for merely continuous CR-maps
between real-analytic hypersurfaces which are of D'Angelo finite
type (meaning they do not contain any complex varieties) in
$\CC{N}$, $N\geq 3$, is contained in the work of Diederich
and Pinchuk \cite{DiPi}.
  For notable results on the reflection principle for CR-mappings between
CR-submanifolds of different dimension see Coupet, Pinchuk and Sukhov \cite{cps}, Meylan, Mir and Zaitsev \cite{mmz} and Mir \cite{mir2003}.
 However, these positive results
do not apply to more degenerate situations, and also do not help to
shed light on the different roles of minimality and nondegeneracy.

For hypersurfaces in $\CC{2}$, the concepts of essential finiteness and minimality 
actually agree, so that violation of either of these conditions leads to the consideration of nonminimal hypersurfaces.
As CR-mappings between Levi flat hypersurfaces can trivially be non-analytic, we restrict the considerations to Levi nonflat hypersurfaces (in $\CC{2}$ the latter property is equivalent to \it holomorphic nondegeneracy, \rm see \cite{ber}).
Easy examples show that
one cannot hope for diffeomorphism of class $C^k$ for finite $k$ to
enjoy the analyticity property in the degenerate setting. For $C^\infty$ smooth CR-diffeomorphisms,
 Ebenfelt \cite{ebenfelt} established  that such diffeomorphisms
 between real-analytic 1-nonminimal hypersurfaces in $\CC{2}$ are analytic.
Recall that, according to Meylan \cite{meylan}, a nonminimal at a
point $p$ real-analytic hypersurface $M\subset\CC{N}$ is called
 \emph{$m$-nonminimal at $p$},  if in some local coordinates,
vanishing at $p$, $M$ can be represented as 
$$\im w=(\re)^mH(z,\bar
z,\re w),\,H(z,\bar z,0)\not\equiv 0.$$ 
Here
$(z,w)\in\CC{N-1}\times\CC{}$ denote the coordinates in $\CC{N}$
and $m\in[1,\infty)$ is an integer, known to be a biholomorphic
invariant of $(M,p)$.
For some notable analyticity results for
CR-mappings between nonminimal hypersurfaces, addmitting one-sided
holomorphic extension, we refer to \cite{meylan,hu1,hu2,hmm}. The most general result in this direction was obtained by Ebenfelt and Huang \cite{eh},
who showed that merely continuous boundary values have the analyticity property,
 as long as $M,M\rq{}$ are Levi nonflat.
However, the general
question whether a smooth CR-diffeomorphism between Levi nonflat hypersurfaces is necessarily the restriction
of a holomorphic map remained open, even in dimension $2$.
Evidence in the algebraic case (see Baouendi, Huang and Rothschild \cite{bhr})
provided some basis for hopes in that direction, and the following was conjectured by
Ebenfelt and Huang.

\begin{conjecture}[see \cite{eh}] \label{thm:ehconj}
Let $M,M\rq{}\subset\CC{2}$ be real-analytic Levi nonflat hypersurfaces. Then any
$C^\infty$-smooth CR (local) diffeomorphism $F:\,M\to M\rq{}$ extends holomorphically to an open
neighborhood of $M$ in $\CC{2}$.
\end{conjecture}

Our main result provides the negative answer to that conjecture: we construct
examples of Levi nonflat hypersurfaces in $\CC{2}$, possessing
 a smooth CR-diffeomorphism between them which is not the restriction of
a holomorphic map.

  In order to discuss our results in more detail,
 let us introduce a number of natural spaces of maps
 between real-analytic CR-manifolds. We will write
 $\diffcr{k} ((M,p),M')$ for the space of germs of CR-diffeomorphisms
 of class $C^k$, where $k\in \mathbb{N}\cup \{ \infty, \omega\}$, and
 $\diffcr{k} ((M,p), (M',p'))$ for those diffeomorphisms $H$ which in
 addition satisfy $H(p) = p'$. We will also need the space of
 {\em formal} CR-diffeomorphisms for which we will write
  $\diffcr{f} ((M,p),M') $ and   $\diffcr{f} ((M,p), (M',p'))$, respectively.
 In the case $M' = M$ we use the notation
 $\Hol{k} (M,p ) = \diffcr{k} ((M,p),M) $ and
 $\Aut{k} (M,p) = \diffcr{k} ((M,p),(M,p)) $, $k\in \mathbb{N}\cup \{ \infty, \omega,f\}$.

Our first main result implies that the conjecture of Ebenfelt and Huang
cited above has the negative answer.

 \begin{theorem}\label{thm:ehwrong}
For any positive integers $n,k>0$ there exist germs of real-analytic holomorphically nondegenerate
 CR-submanifolds  $(M,p)$, $(M',p')$ in  $\CC{n+k}$ of CR-dimension $n$ and CR-codimension $k$
such that
$$\diffcr{\infty} ((M,p),(M',p')) \neq \emptyset,\,\,\, \mbox{but}\,\,\,\,
\diffcr{\omega} ((M,p),(M',p')) = \emptyset .$$
 \end{theorem}

An immediate crucial corollary from \autoref{thm:ehwrong} is that, in any positive CR-dimension and CR-codimension, the holomorphic and the $C^\infty$ CR equivalence problems are \it distinct. \rm To formulate this corollary in detail, we fix two integers $n,k\geq 0$ and introduce the \it $C^\infty$ CR moduli space \rm $\mathfrak M_{\infty}^{n,k}$ and the
 \it holomorphic moduli space \rm $\mathfrak M_\omega^{n,k}$
 as the space of $C^\infty$ CR-equivalence
classes and the space of biholomorphic
equivalence classes for germs  of real-analytic
CR-submanifolds in $\CC{n+k}$ of CR-dimension $n$ and
CR-codimension $k$ at the origin, respectively. We have the natural surjective map
$\mathfrak i_{n,k}:\,\mathfrak M_\omega^{n,k} \to \mathfrak M_{\infty}^{n,k}$.

\begin{corollary}\label{thm:modulispace}
 For any integers $n,k>0$ the map $\mathfrak i_{n,k}:\,\mathfrak M_\omega^{n,k} \to \mathfrak M_{\infty}^{n,k}$ is  not injective.
\end{corollary}

\medskip

Thus, in any positive CR-dimension and CR-codimension, the holomorphic moduli space
of germs at the origin of real-analytic CR-submanifolds is \it bigger \rm than the corresponding $C^\infty$ CR moduli space.

 Setting in \autoref{thm:ehwrong} $n=k=1$, we immediately obtain the negative answer to
 \autoref{thm:ehconj}.
We note that examples of non-analytic $C^\infty$ smooth CR-mappings between
Levi nonflat hypersurfaces in $\CC{2}$ were previously obtained by
Ebenfelt \cite{ebenfeltold}, however, these mappings all vanish to
infinite order at $0$ and thus do not fall into the category of
CR-diffeomorphisms.

We note that \autoref{thm:ehwrong} also implies that, in the nonminimal case,
the approximation property for CR-equivalences between real-analytic submanifolds
$M,M'\subset\CC{N}$ akin to the Baouendi-Treve's property \cite{bt}
 of
CR-functions or CR Artin\rq{}s Approximation Property for CR-mappings (see
Mir \cite{mirapprox} and Sunye \cite{sunye}) fails. 

\begin{corollary}\label{cor:noapprox}
 For any integers $n,k>0$ there exist real-analytic
CR-submanifolds $M,M'\subset\CC{n+k}$ of CR-dimension $n$ and
CR-codimension $k$ and a $C^\infty$ CR-diffeomorphism
$F:\,(M,p)\lr (M',p')$ which, for any fixed open set
$U\subset\CC{n}$, can not be approximated by holomorphic mappings
$M\cap U\lr M'$; its formal Taylor series also cannot be approximated by
holomorphic series taking $M$ into $M'$.
 \end{corollary}

This result shows  that
$\diffcr{\infty} ((M,p),(M',p'))$ is in general not an
appropriate ``closure'' of $\diffcr{\omega} ((M,p),(M',p'))$.

% \noindent {\bf Conjecture 1} (\cite{eh}). Let $M,M'\subset\CC{2}$
% be real-analytic Levi nonflat hypersurfaces. Then any $C^\infty$
% CR-equivalence $F:\,(M,p)\lr (M',p'),\,p\in M,p'\in M'$ is
% analytic.

% \medskip
%We also remark the following.
%For two fixed integers $n,k\geq 0$ we denote by, respectively,
%$\mathfrak M_{\infty}^{n,k}$  the space of $C^\infty$ CR-equivalence
%classes and by $\mathfrak M_\omega^{n,k}$ the space of biholomorphic
%equivalence classes for germs  of real-analytic
%CR-submanifolds in $\CC{n+k}$ of CR-dimension $n$ and
%CR-codimension $k$ at the origin, i.e. we identify
%germs $(M,0)$ and $(M',0)$ if
%$\diffcr{\infty} ((M,0), (M',0) ) \neq \emptyset$ and
%$\diffcr{\omega} ((M,0), (M',0) ) \neq \emptyset$, respectively.
%Since $\diffcr{\infty} ((M,0), (M',0) ) \supset\diffcr{\omega} ((M,0), (M',0) ) \neq \emptyset$,
%every $C^\infty$ CR-equivalence class splits up into the biholomorphic equivalence classes.
%However, the natural map $\mathfrak M_\omega^{n,k} \to \mathfrak M_{\infty}^{n,k}$ assigning
%the smooth equivalence class to the biholomorphic equivalence class is \emph{not injective}, as
%\autoref{thm:ehwrong} implies.

It is then natural to ask whether analyticity results  hold
 for CR-automorphisms of holomorphically nondegenerate
CR-manifolds, i.e., whether the groups $\Aut{\infty} (M,p)$
and $\Aut{\omega} (M,p)$ coincide for a germ of a
real-analytic CR-submanifold $(M,p)$. Our next result shows that the
answer is also negative, even for the infinitesimal
automorphism algebras. Recall that the {\em infinitesimal
automorphism  algebra }  for a
real submanifold $M\subset\CC{N}$ at a point $p\in M$ \rm is the
algebra $\hol{k} (M,0)$
of holomorphic ($k=\omega$)  or smooth ($k=\infty$) vector fields
\[X=f_1\frac{\partial}{\partial z_1}+...+f_n\frac{\partial}{\partial z_N},\] defined near $p$
such that each $f_j$ is a real-analytic ($k=\omega$) or smooth ($k=\infty$)
CR-function on $M$ and
 $X+\bar X$ is tangent to $M$ near $p$. Vector fields $X\in\mathfrak{hol}\,(M,0)$ (resp.
$X\in\mathfrak{hol}^\infty(M,0)$) are exactly the vector fields
generating flows of holomorphic (resp. smooth CR)
transformations, preserving $M$ locally. The  \emph{stability
subalgebras}
$\aut{k}(M,0)\subset\hol{k} (M,0)$ are
 determined by the condition $X|_p=0$.

\begin{theorem}\label{thm:infinitesimal}
For any integer $N\geq 2$ there exist real-analytic
holomorphically nondegenerate hypersurfaces $M\subset\CC{N}$,
$M\ni 0$, with
 $\hol{\omega} (M,0)\subsetneq
\hol{\infty} (M,0)$ and $\aut{\omega} (M,0)\subsetneq
\aut{\infty} (M,0)$.
\end{theorem}

\autoref{thm:infinitesimal}, read together with the results in \cite{divergence}, poses an interesting problem
of finding the relations between the, respectively, holomorphic, CR and formal stability algebras $\aut{\omega}(M,0),\aut{\infty}(M,0)$ and $\aut{f}(M,0)$ for a real-analytic nonminimal Levi nonflat hypersurface $M\subset\CC{2}$. Note that  the results in \cite{ebenfelt} and \cite{jl2} show that the three algebras coincide
in the case of $1$-nonminimal hypersurfaces. We also point out that a recent result of Shafikov and the first author in \cite{nonminimalODE} provides the sharp upper bound $\mbox{dim}\,\aut{\omega}(M,0)\leq 5$ for an arbitrary Levi nonflat real-analytic hypersurface $M\subset\CC{2}$. However, no known results imply the same bound for the algebras $\aut{\infty}(M,0)$ and $\aut{f}(M,0)$.  This motivates the following two open problems.

\begin{problem} Establish optimal regularity conditions for a real-analytic nonminimal Levi nonflat hypersurface $M\subset\CC{2}$, generalizing the $1$-nonminimality and guaranteeing the coincidence of the algebras $\aut{\omega}(M,0),\aut{\infty}(M,0)$ and $\aut{f}(M,0)$.
\end{problem}

 \begin{problem} Find the sharp upper bound for the dimension of the algebras $\aut{\infty}(M,0)$ and $\aut{f}(M,0)$ for a
real-analytic Levi nonflat hypersurface $M\subset\CC{2}$.
\end{problem}

The main tool of the paper is a development of a recent
CR \,$\lr$\,DS (Cauchey-Riemann manifolds \,$\lr$\,\,Dynamical Systems) technique introduced by Shafikov and the first
author \cite{divergence,nonminimalODE}. The technique suggests to
replace a given CR-submanifold $M$ with a CR-degeneracy (such as
nonminimality) by an appropriate holomorphic dynamical system
$\mathcal E(M)$, and then study mappings of CR-submanifolds
accordingly. This method previously enabled to show
\cite{divergence}
 that, in any positive CR-dimension and CR-codimension, there are more holomorphic moduli for real-analytic
CR-submanifolds than formal ones (compare with the result in
\cite{brz}).  The possibility to replace a real-analytic
CR-manifold by a complex dynamical system is based on the
fundamental connection between CR-geometry and the geometry of
completely integrable PDE systems, first observed by E.~Cartan and
Segre \cite{cartan,segre}, and recently revisited in the work of
Sukhov \cite{sukhov1,sukhov2} (see also \cite{gm,nurowski} for
some further properties of the connection). The ``mediator''
between a CR-manifold and the associated PDE system is the Segre
family of the CR-manifold. By choosing real hypersurfaces
$M,M'\subset\CC{2}$ in such a way that mappings between the
associated dynamical systems $\mathcal E(M),\mathcal E(M')$ have
certain ``wedge''-type regularity, but are not regular in an open
neighborhood of the singular point, we obtained the desired
counterexamples.

We shall  also note that the paper contains an important
intermediate result which is a  {\em complete  characterization}
of all real-analytic hypersurfaces in
$\CC{2}$, which are nonminimal at the origin and
spherical outside the complex locus $X\ni 0$ (see
\autoref{thm:classify} and \autoref{cor:algorithm} below). The latter class of
hypersurfaces was previously studied in a long sequence of
publications
\cite{kowalski,elz,belnew,kl,nonminimal,divergence,nonminimalODE}
and appears to be highly nontrivial. The results of Section 3
below completes the study of hypersurfaces of this class.

We briefly describe the structure of the paper. In Section 2 we
provide necessary background information. In Section 3 we
establish a class of singular meromorphic complex differential
equations that are associated with a class of nonminimal
hypersurfaces in $\CC{2}$ (namely, the class of nonminimal
hypersurfaces, spherical outside the complex locus). We call them
 \emph{ODEs with a real structure}  (compare with the work
\cite{faran} of Faran, where Segre families with a real structure
were studied). This gives us a freedom in choice of nonminimal
hypersurfaces, for which the associated ODEs have prescribed
properties. We also obtain in the same section the above mentioned
characterization theorem for nonminimal spherical hypersurfaces.
In Section 4 we provide a one-parameter family $\mathcal E_\gamma$
of ODEs with a real structure, any two of which are equivalent by
means of a \emph{ sectorial} transformation, while each ODE
$\mathcal E_\gamma$ is inequivalent to $\mathcal E_0$
holomorphically for $\gamma\neq 0$. Remarkably, all ODEs $\mathcal
E_\gamma$ have trivial monodromy of solutions. It follows
immediately that a real hypersurface $M_\gamma$ behind an ODE
$\mathcal E_\gamma$ with $\gamma\neq 0$ is holomorphically
inequivalent to $M_0$, and the rest of the section is dedicated to
the proof of the fact that all $M_\gamma$ are sectorially
equivalent. For that we introduce and use the class of so-called
\emph{ sectorial coupled gauge transformation}.  It is not
difficult then to deduce the proof of \autoref{thm:ehwrong}. In Section 5 we
apply the non-analytic near the origin sectorial mapping of
$M_\gamma$ into $M_0$ to describe the Lie algebras
$\mathfrak{hol}^\omega\,(M_\gamma,0),
\mathfrak{hol}^\infty(M_\gamma,0),\mathfrak{aut}^\omega\,(M_\gamma,0),
\mathfrak{aut}^\infty(M_\gamma,0)$ for $\gamma\neq 0$ and deduce
from there the proof of \autoref{thm:infinitesimal}.

\smallskip

\begin{center}\bf Acknowledgments \end{center}

\smallskip

 We would like to thank Nordine Mir for multiple helpful comments on the initial version of the paper.
Both authors are supported by the Austrian Science Fund (FWF). 

\smallskip

\section{Preliminaries}
\label{sec:prelim}

\subsection{Segre varieties.}
Let $M$ be a smooth  real-analytic submanifold in $\cx^{n+k}$ of
CR-dimension $n$ and CR-codimension $k$, $n,k>0$, $0\in M$, and
$U$ a neighbourhood of the origin where $M\cap U$ admits a
real-analytic defining function $\phi(Z,\overline Z)$ with
the property that $\phi (Z,\zeta)$ is a holomorphic function for
for $(Z, \zeta) \in U\times \bar U$. For every
point $\zeta\in U$ we associate its Segre
variety in $U$ by
$$
Q_\zeta= \{Z\in U : \phi(Z,\overline \zeta)=0\}.
$$
Segre varieties depend holomorphically on the variable $\overline
\zeta$, and for small enough neighbourhoods $U$ of $0$, they
are actually holomorphic submanifolds of $U$ of codimension $k$.

One can choose coordinates $Z = (z,w) \in \CC{n} \times \CC{k} $ and
 a neighbourhood
% Note I dropped U_1 here since otherwise 2.1 is technically incorrect and it is of no use here really.
$U={\
U^z}\times U^w\subset \cx^{n}\times \CC{k}$ such that, for any $\zeta\in U,$
$$
  Q_\zeta=\left \{(z,w)\in U^z \times U^w: w = h(z,\overline \zeta)\right\}
$$
is a closed complex analytic graph. $h$ is a holomorphic
function on $U^z \times \bar U$.  The antiholomorphic $(n+k)$-parameter family of complex
submanifolds $\{Q_\zeta\}_{\zeta\in U_1}$ is called  {\em the Segre
family} of $M$ at the origin. The following basic
properties of Segre varieties follow from the definition and the
reality condition on the defining function:
\begin{equation}\label{e.svp} \begin{aligned}
  Z\in Q_\zeta & \Leftrightarrow  \zeta\in Q_Z,
 \\
  Z\in Q_Z & \Leftrightarrow  Z\in M,
\\
 \zeta\in M & \Leftrightarrow \{Z\in U \colon Q_\zeta=Q_Z\}\subset M.
\end{aligned}
\end{equation}

The fundamental role of Segre varieties for holomorphic maps
is due to their {\em invariance property}: If $f: U \to U'$ is a
holomorphic map which sends a smooth real-analytic submanifold
$M\subset U$ into another such submanifold $M'\subset U'$, and $U$
is chosen as above (with the analogous choices and
notations for $M'$), then
$$
 f(Q_Z)\subset Q'_{f(Z)}.
$$
For
more details and other properties of Segre varieties we refer the reader
to
e.g. \cite{webster}, \cite{DiFo},\cite{DiPi},  or \cite{ber}.

A  particularly important case arises when $M$ is a {\em real
hyperquadric},  i.e., when
$$M=\left\{
[\zeta_0,\dots,\zeta_N]\in \cx\mathbb P^{N} : H(\zeta,\bar \zeta)
=0  \right\},$$
where $H(\zeta,\bar \zeta)$ is a nondegenerate
Hermitian form on  $\CC{N+1}$ with $k+1$ positive and $l+1$
negative eigenvalues, $k+l=N-1,\,0\leq l\leq k\leq N-1$. In that
case, the Segre
variety of a point $\zeta \in\CP{N}$ is
 the globally defined projective hyperplane
$ Q_\zeta = \{\xi\in \cx\mathbb P^N: H(\xi,\bar\zeta)=0\}$, and the
Segre family $\{Q_\zeta,\,\zeta\in\CP{N}\}$  coincides in this
case with the space $(\CP{N})^*$ of all projective hyperplanes in
$\CP{N}$.

The space of Segre varieties $\{Q_Z : Z\in U\}$,
for appropriately chosen $U$, can be
identified with a subset of $\cx^K$ for some $K>0$ in such a way
that the so-called {\em Segre map} $\lambda : Z \to Q_Z$ is
holomorphic. This  can be seen from the fact that if we write
\[ h(z,\bar \zeta) = \sum_{\alpha\in\mathbb{N}^n} h_\alpha (\bar \zeta) z^\alpha, \]
then $\lambda (Z)$ can be identified with
$\left(h_\alpha (\bar Z) \right)_{\alpha\in\mathbb{N}^n}$. After that the desired fact follows from the Noetherian property.

If $M$ is a hypersurface,
then its Segre map is one-to-one in a
neighbourhood of every point $p$ where $M$ is Levi nondegenerate.
When such a real hypersurface $M$ contains a complex hypersurface
$X$, for any point $p\in X$ we have $Q_p = X$ and $Q_p\cap
X\neq\emptyset\Leftrightarrow p\in X$, so that the Segre map
$\lambda$ sends the entire $X$ to a unique point in $\CC{N}$ and,
accordingly, $\lambda$ is not even finite-to-one near each $p\in
X$ (i.e., $M$ is \it not essentially finite \rm at points $p\in
X$). If $\Q\subset\CP{N}$ is  a hyperquadric,  its Segre map $\lambda'$
is the global natural one-to-one correspondence between $\CP{N}$ and
the space $(\CP{N})^*$ given by the polar construction.

%%%%%%%%%%%%%%%%%%%%%%%%%%%%%%%%%%
\subsection{Real hypersurfaces and second order differential equations.}\label{sub:realhyp2ndorderequ}
To every Levi nondegenerate real hypersurface
$M\subset\CC{N}$  we can associate a system of second order
holomorphic PDEs with $1$ dependent and $N-1$ independent
variables, using the Segre family of the hypersurface. This remarkable construction
 goes back to
E.~Cartan \cite{cartanODE},\cite{cartan} and Segre \cite{segre},
and was recently revisited in
\cite{sukhov1},\cite{sukhov2},\cite{nurowski},\cite{gm} (see also
references therein).
% STOP 3/20

Let us  describe this procedure in the case
$N=2$ relevant for our purposes.
We denote the coordinates in $\CC{2}$ by $(z,w)$, and put
$z=x+iy,\,w=u+iv$. Let $M\subset\CC{2}$ be a smooth real-analytic
hypersurface, passing through the origin, and choose $U = U_z \times U_w$
 as described above. In this case
we associate a second order holomorphic ODE to $M$, which is uniquely determined by the condition that the equation is satisfied by all the
graphing functions $h(z,\zeta) = w(z)$ of the
Segre family $\{Q_\zeta\}_{\zeta\in U}$ of $M$ in a
neighbourhood of the origin.

More precisely, since $M$ is Levi-nondegenerate
near the origin, the Segre map
$\zeta\lr Q_\zeta$ is injective and the Segre family has
the so-called transversality property: if two distinct Segre
varieties intersect at a point $q\in U$, then their intersection
at $q$ is transverse. Thus, $\{Q_\zeta\}_{\zeta\in U}$ is a
2-parameter  family of holomorphic
curves in $U$ with the transversality property, depending
holomorphically on $\bar\zeta$. It follows from
the holomorphic version of the fundamental ODE theorem (see, e.g.,
\cite{ilyashenko}) that there exists a unique second order
holomorphic ODE $w''=\Phi(z,w,w')$, satisfied by all the graphing functions of
$\{Q_\zeta\}_{\zeta\in U}$.

To be more explicit we consider the
so-called {\em  complex defining
 equation } (see, e.g., \cite{ber})\,
$w=\rho(z,\bar z,\bar w)$ \, of $M$ near the origin, which one
obtains by substituting $u=\frac{1}{2}(w+\bar
w),\,v=\frac{1}{2i}(w-\bar w)$ into the real defining equation and
applying the holomorphic implicit function theorem. The complex
defining function $\rho$ of a real hypersurface satisfies the {\em reality
condition}
\begin{equation}\label{reality}
w\equiv\rho(z,\bar z,\bar\rho(\bar z,z,w)).
\end{equation}
We shall again assume that $U$ is a neighbourhood of the origin
chosen as above.
 The Segre
variety $Q_p$ of a point $p=(a,b)\in U$ is  now given
as the graph
\begin{equation} \label{segre}w (z)=\rho(z,\bar a,\bar b). \end{equation}
Differentiating \eqref{segre} once, we obtain
\begin{equation}\label{segreder} w'=\rho_z(z,\bar a,\bar b). \end{equation}
Considering \eqref{segre} and \eqref{segreder}  as a holomorphic
system of equations with the unknowns $\bar a,\bar b$, an
application of the implicit function theorem yields holomorphic functions
 $A, B$ such that
$$
\bar a=A(z,w,w'),\,\bar b=B(z,w,w').
$$
The implicit function theorem applies here because the
Jacobian of the system coincides with the Levi determinant of $M$
for $(z,w)\in M$ (\cite{ber}). Differentiating \eqref{segre} twice
and substituting for $\bar a,\bar b$ finally
yields
\begin{equation}\label{segreder2}
w''=\rho_{zz}(z,A(z,w,w'),B(z,w,w'))=:\Phi(z,w,w').
\end{equation}
Now \eqref{segreder2} is the desired holomorphic second order ODE
$\mathcal E = \mathcal{E}(M) $.

More generally, the association of   a completely integrable PDE  with
a CR-manifold is possible for a wide range of
CR-submanifolds (see \cite{sukhov1,sukhov2,gm}). The
correspondence $M\lr \mathcal E(M)$ has the following fundamental
properties:

\begin{enumerate}

\item[(1)] Every local holomorphic equivalence $F:\, (M,0)\lr (M',0)$
between CR-submanifolds is an equivalence between the
corresponding PDE systems $\mathcal E(M),\mathcal E(M')$ (see \autoref{sub:equiv2ndorder});
%We should clarify here what we mean by equivalence of PDE systems
\item[(2)] The complexification of the infinitesimal automorphism algebra
$\mathfrak{hol}^\omega(M,0)$ of $M$ at the origin coincides with the Lie
symmetry algebra  of the associated PDE system $\mathcal E(M)$
(see, e.g., \cite{olver} for the details of the concept).

\end{enumerate}

We emphasize here that if
$M\subset\CC{2}$ is a real hypersurface which is nonminimal at the origin, there is a priori {\em no} way to associate
to $M$ a second order ODE or even a more general PDE system near
the origin. However, in \cite{nonminimalODE} the authors
discovered an injective correspondence between real hypersurfaces which are nonminimal at the
origin and spherical outside the complex locus hypersurfaces
$M\subset\CC{2}$ and certain {\em singular} complex ODEs $\mathcal E(M)$ with an
isolated meromorphic singularity at the origin. In Section 3 we
complete the study initiated in  \cite{nonminimalODE} by finding a
precise description of the image for the above injective
correspondence.

%NOTE: Is correspondence a good word here - it has a number of connotations? Should one maybe use functor notation and say faithful instead?
%%%%%%%%%%%%%%%%%%%%%%%%%%%%%%%%%%%

%%%%%%%%%%%%%%%%%%%%%%%%%%%%

\subsection{Equivalence problem for second order ODEs}\label{sub:equiv2ndorder}
We start with a description of the jet prolongation approach to
the equivalence problem (which is a simple interpretation of a
more general approach in the context of {\em jet bundles}).
In what follows all variables are assumed to be complex, all
mappings biholomorphic, and all ODEs to be defined near their zero
solution $y(x)=0$.

Consider two ODEs,  $\mathcal E$ given by $y^{''}=\Phi(x,y,y')$
and
$\tilde{\mathcal E}$ given by $ y^{''}=\tilde\Phi(x,y,y')$, where the functions
$\Phi$ and $\tilde\Phi$ are holomorphic in some neighbourhood of the
origin in $\CC{3}$. We say that a germ  of a biholomorphism
$F \colon (\CC{2},0)\lr(\CC{2},0)$  transforms $\mathcal E$ into
$\tilde{\mathcal E}$, if it sends (locally) graphs of solutions of
$\mathcal E$ into graphs of solutions of $\tilde{\mathcal E}$.
We define the {\em $2$-jet space} $J^{(2)}$ to be a $4$-dimensional
 linear space with coordinates $x,y,y_1,y_2$,
which correspond to the independent variable $x$, the dependent
variable $y$ and its derivatives up to order $2$, so that we can
naturally consider $\mathcal E$ and $\tilde{\mathcal E}$ as complex
submanifolds of $J^{(2)}$.

For any biholomorphism $F$ as above one may consider its
{\em $2$-jet prolongation} $F^{(2)}$, which is defined on a neighbourhood of the origin in $\CC{4}$ as follows.
The first two components of the mapping $F^{(2)}$
coincide with those of $F$. To obtain the remaining components  we
denote the coordinates in the preimage by $(x,y)$ and in the
target domain by $(X,Y)$. Then the derivative $\frac{dY}{dX}$ can
be symbolically recalculated, using the chain rule, in terms of
$x,y,y'$, so that the third coordinate $Y_1$ in the target jet
space becomes a function of $x,y,y_1$. In the same manner one
obtains the fourth component of the prolongation of the
mapping $F$. Thus
the mapping $F$ transforms the ODE $\mathcal E$ into $\tilde{\mathcal
E}$ if and only if the prolonged mapping $F^{(2)}$ transforms
$(\mathcal E,0)$ into $(\tilde{\mathcal E},0)$ as submanifolds in the
jet space $J^{(2)}$. A similar statement can be formulated for
certain singular differential equations, for example, for linear
ODEs (see, e.g., \cite{ilyashenko}).

The local equivalence problem for (nonsingular!) second order ODEs
was solved in the celebrated papers of E.~Cartan \cite{cartanODE}
and A.~Tresse \cite{tresse}. We briefly describe below Tresse's
approach, as it is of particular importance for us. A
\emph{semi-invariant}  for the action of the group
$\mbox{Diff}(\CC{2},0)$ of  biholomorphisms of $(\CC{2},0)$
on the space of germs at the origin of right-hand sides
$\Phi(x,y,y_1)$ of second order holomorphic ODEs \,
$y''=\Phi(x,y,y')$\, is a differential-algebraic polynomial
$L(\Phi(x,y,y_1))$ such that its value $L(\tilde\Phi(X,Y,Y_1))$ at the
transformed  ``point'' $\tilde\Phi(X,Y,Y_1)$ differs from the initial
value $L(\Phi(x,y,y_1))$ by a
 factor $\lambda(x,y,y_1)$ non-vanishing near the origin.

 In  \cite{tresse} Tresse found the complete system of
semi-invariants for the equivalence problem for 2nd order ODEs. In particular, he
found the two basic (lowest order) semi-invariants
\begin{gather} \label{invariants} L_1(\Phi)=
\Phi_{y_1y_1y_1y_1}\\
\notag L_2(\Phi)=D^2\Phi_{y_1y_1}-4D\Phi_{yy_1}-\Phi_{y_1}\cdot
D\Phi_{y_1y_1}+4\Phi_{y_1}\Phi_{yy_1}-3\Phi_y\Phi_{y_1y_1}+6\Phi_{yy},\end{gather}
where the differential operator $D$ is defined by
 \[ D:=\frac{\partial}{\partial x}+y_1\frac{\partial}{\partial
y}+\Phi\frac{\partial}{\partial y_1}. \]

A second order ODE is locally equivalent to the {\em flat}
(or   {\em simplest}) ODE $Y''=0$ if and only if the two
basic invariants vanish: $$L_1(\Phi)=L_2(\Phi)=0.$$ The
concept of the {\em dual second order ODE} connects  the two basic invariants. For the family of solutions $\mathcal S=\bigr\{y=\Phi(x,\xi,\eta)\bigl\}_{\xi,\eta\in(\CC{2},0)}$ of a second order ODE $\mathcal
E:\,y''=\Phi(x,y,y')$, considered near the zero solution $y=0$,  the
two-parameter family $\mathcal S^*$, given by the implicit
equation $\eta=\Phi(\xi,x,y)$, is called  {\em dual} for
$\mathcal S$. The unique second order ODE $\mathcal E^*$,
satisfied by the family $\mathcal S^*$ (see \autoref{sub:realhyp2ndorderequ}), is
called {\em dual}  for $\mathcal E$. A dual ODE is not
unique, as it depends on the parametrization of the family
$\mathcal S$, but its equivalence class with respect to the action of
$\mbox{Diff}(\CC{2},0)$ is unique and well defined. Remarkably,
for any choice of the dual ODE $\mathcal
E^*=\{y''=\Phi^*(x,y,y_1)\}$ there exist two non-vanishing near
the origin factors $\lambda(x,y,y_1),\mu(x,y,y_1)$ such that
$$L_1(\Phi)=\lambda \cdot L_2(\Phi^*), \quad L_2(\Phi)=\mu \cdot L_1(\Phi^*).$$
In particular,  $\mathcal E$ is locally equivalent to the simplest ODE if and only if
both $\mathcal E$ and $\mathcal E^*$ are cubic with respect to $y_1$.

For a modern treatment of the problem and some further
developments we refer to the book of V.~Arnold \cite{arnoldgeom},
and also to the work of B.~Kruglikov \cite{kruglikov} and
P.~Nurowski and G.~Sparling \cite{nurowski}.

%%%%%%%%%%%%%%%%%%%%%%%
\subsection{Complex linear differential equations with an isolated singularity}
% Perhaps the most important and geometrical class of complex
% differential equations is the class of complex linear ODEs.
% That depends on the point of view. Maybe:
Complex linear ODEs are important classical objects, whose geometric
interpretations are plentiful. We
refer to the excellent sources \cite{ilyashenko}, \cite{ai},
\cite{bolibruh}, \cite{vazow},\cite{coddington}  on  complex linear differential equations, gathering here the facts that we will need in
the sequel.

A {\em first order linear system} of $n$ complex ODEs in a domain
$G\subset\CC{}$ (or simply a {\em linear system} in a domain $G$)
is a holomorphic ODE system $\mathcal L$ of the form
$y'(w)=A(w)y(w)$, where $A(w)$ is a
holomorphic in $G$
 function, taking values in the space of  $n\times n$ matrices,
 and $y(w)=(y_1(w),...,y_n(w))$
is an $n$-tuple of unknown functions. Solutions of $\mathcal L$
near a point $p\in G$ form a linear space of dimension~$n$.
Moreover, any germ of a solution near a point $p\in G$ of $\mathcal L$
extends analytically along any path $\gamma\subset G$, starting
at~$p$, so that any solution $y(w)$ of $\mathcal L$ is defined
globally in $G$ as a (possibly multiple-valued) analytic function.
 A \it fundamental system of solutions for $\mathcal L$ \rm
is a matrix whose columns form some collection of $n$ linearly
independent solutions of $\mathcal L$.

If $G$ is a punctured disc, centered at $0$, we say
that $\mathcal L$  is a system {\em with an isolated singularity at $w=0$}.
An important (and sometimes even a complete) characterization
of an isolated singularity is its {\em  monodromy operator}, which is
defined as follows. If $Y(w)$ is some fundamental system of
solutions of $\mathcal L$ in $G$, and $\gamma$ is a simple loop
about the origin, then it is not difficult to see that the
monodromy of $Y(w)$ with respect to $\gamma$ is given by the right
multiplication by a constant nondegenerate matrix $M$, called the
{\em monodromy matrix}.  The matrix $M$ is defined up to a
similarity, so that it defines a linear operator
$\CC{n}\lr\CC{n}$, which is  called the monodromy operator of the
singularity.

If $A(w)$ has a pole at the isolated singularity $w=0$,
%want to include removable???
we say that the system has a {\it meromorphic singularity}. As
the solutions of $\mathcal L$ are holomorphic in any proper sector
$S\subset G$ of a sufficiently small radius with vertex at
$w=0$, it is important to study the behaviour of the solutions as
$w\rightarrow 0$. If all solutions of $\mathcal L$ admit a bound
$||y(w)||\leq C|w|^a$  in any such sector (with some constants
$C>0,\ a\in \mathbb R$, depending possibly on the sector), then
$w=0$ is a {\em regular singularity}, otherwise it is
 an {\em irregular singularity}.  In particular, if the monodromy is trivial, then the singularity is regular if and
only if all the solutions of $\mathcal L$ are meromorphic in $G$.

L.~Fuchs introduced the following condition: the singular point
$w=0$ is {\em Fuchsian}, if $A(w)$ is meromorphic  at
$w=0$ and has a pole of order $\leq 1$ there. The Fuchsian
condition turns out to be sufficient for the regularity of a
singular point. Another remarkable property of Fuchsian
singularities can be described as follows. We say that
two complex
linear systems with an isolated singularity $\mathcal L_1,\mathcal
L_2$ are {\em (formally) equivalent},  if there exists a (formal)
transformation $F:\,(\CC{n+1},0)\lr(\CC{n+1},0)$ of the form
$F(w,y)=(w,H(w)y)$ for some (formal) invertible matrix-valued
function $H(w)$, which transforms (formally) $\mathcal L_1$ into
$\mathcal L_2$. It turns out that two Fuchsian systems are
formally equivalent if and only if they are holomorphically equivalent
 (in fact, any formal equivalence between them as
above must be convergent). Any Fuchsian system can be brought to a
special polynomial  form (in the sense that the matrix $wA(w)$ is
polynomial) called the {\em Poincare-Dulac normal form for Fuchsian
systems}, and moreover, the normalizing transformation is
always convergent.

However, in the {\em non}-Fuchsian case the behavior of solutions
and mappings between  linear systems is totally different.
Generically, solutions of a non-Fuchsian system
\[y'=\frac{1}{w^m}B(w)y, \quad m\geq 2\]  do {\em not } have polynomial
growth in sectors, and formal equivalences between non-Fuchsian
systems are divergent, as a rule. Also the transformation
bringing a non-resonant non-Fuchsian system to a special
polynomial form called {\em Poincare-Dulac normal form for
non-Fuchsian systems}  is usually also divergent. As some
compensation for this divergence phenomenon, we formulate below a
remarkable result,  {\em Sibuya's sectorial normalization
theorem},   which is of fundamental importance for our
constructions.

For a system $y'=\frac{1}{w^m}B(w)y,\,m\geq 2$
which is non-resonant (i.e., the leading matrix $B_0=B(0)$ has
pairwise distinct eigenvalues $\{\lambda_1,...,\lambda_n\}$) we
call each of the $2(m-1)$ rays
$R_{ij}=\left\{\re\left((\lambda_i-\lambda_j)w^{1-m}\right)=0\right\},\,i,j=1,...,n,\,i\neq
j,$  a {\em separating ray for the system}. Recall that for a function $f(w)$, holomorphic in a sector $S$
with the vertex at $0$, a formal series \[\hat
f(w)=\sum_{j\geq 0}c_jw^j\] {\em  represents $f(w)$ in $S$
asymptotically}  (one uses the notation $f(w)\sim \hat f(w)$),
 if for every $k\geq 0$
\[\frac{1}{w^k}\left(f(w)-\sum\limits_{j=0}^k
c_jw^j\right)\longrightarrow 0, \quad  w\to 0,w\in S.\] We
refer to \cite{vazow} for further details and properties.

\begin{theorem} [Y.~Sibuya, 1962, see
\cite{sibuya},\cite{ilyashenko}]\label{thm:sibuya} Assume that  a non-Fuchsian
linear system $\mathcal
E$
\[y'=\frac{1}{w^m}B(w)y,\quad m\geq 2\]
is non-resonant
 and $S \subset (\CC{}, 0)$
is an arbitrary sector with vertex at $0$ not containing two
separating rays for any pair of the eigenvalues. Then for any
formal conjugacy $w\mapsto w,\,y\mapsto \hat H(w)y$, conjugating
the system with its Poincare-Dulac polynomial normal form, there
exists a holomorphic  function
$H_S(w)$ defined in $S$ and taking values in $\mbox{GL}(n,\CC{})$ such
that
$H_S(w)$ asymptotically represents $\hat H(w)$
in $S$ and $w\mapsto w,\,y\mapsto H_S(w)y$ conjugates $\mathcal E$
with its Poincare-Dulac normal form in $S$. If a sector $S$ has
opening bigger than $\frac{\pi}{m-1}$, then the sectorial
normalization $H_S(w)$ is unique.
\end{theorem}

Alternatively, one can require for the uniqueness in Sibuya\rq{}s theorem that the sector $S$ contains a separating ray for each pair of eigenvalues of the leading matrix.

We note
that the holomorphic sectorial normalization in \autoref{thm:sibuya} does usually
{\em not} extend to one holomorphic near the origin. The reason is that, somewhat surprisingly, the sectorial
normalization $H_S(w)$ might change from sector to sector by means
of multiplication by a constant matrix $C\in \mbox{GL}(n,\CC{})$
called a  {\em Stokes matrix}.  This phenomenon is known as
the {\em Stokes phenomenon},  and the entire collection $\{C_{ij}\}$ of
Stokes matrixes, corresponding to all separating rays, is called
the {\em Stokes collection}. Generically this collection is non-trivial
(i.e., contains non-identical matrixes). Actually, the
Stokes phenomenon is the conceptual reason for the irregularity
phenomena demonstrated in this paper.

A scalar linear complex ODE of order $n$ in a
 domain $G\subset\CC{}$ is an ODE $\mathcal E$ of the form
$$z^{(n)}=a_{n}(w)z+a_{n-1}(w)z'+...+a_1(w)z^{(n-1)},$$ where $\{a_1(w), \dots a_n (w)\}$ is
a given collection of holomorphic functions in $G$ and $z(w)$ is
the unknown function. By a reduction of $\mathcal E$ to a first
order linear system (see the above references and
also~\cite{vyugin} for various approaches of doing that) one can
naturally transfer most of the definitions and facts, relevant to
linear systems, to scalar equations of order $n$. The main
difference here is contained in the appropriate definition of
Fuchsian: a singular point $w=0$ for an ODE $\mathcal E$ is said to be
{\em Fuchsian},  if the orders of poles $p_j$ of the functions
$a_j(w)$ satisfy the inequalities $p_j\leq j$, $j=1,2,\dots,n$. It
turns out that the condition of Fuchs becomes also necessary for
the regularity of a singular point in the case of $n$-th order
scalar ODEs.

Further information on the classification of isolated
singularities (including details of Poincare-Dulac normalizations
in the Fuchsian and non-Fuchsian cases respectively) can be found
in \cite{ilyashenko}, \cite{vazow} or \cite{coddington}.

\section{Characterization of nonminimal spherical hypersurfaces}\label{sec:nonminimal}

In this section we establish a class of (in general nonlinear)
 second order complex ODEs with a meromorphic
 singularity, which correspond to  real hypersurfaces in
 $\CC{2}$ which are nonminimal at the origin and spherical
 in the complement of their complex locus. Using the connection between hypersurfaces and ODEs,
 this finally gives a  complete description of nonminimal hypersurfaces, spherical in the complement
 to the complex locus. We start with necessary definitions and denote by $\Delta_\varepsilon$ a disc in $\CC{}$, centered at $w=0$
of radius $\varepsilon$, and by $\Delta^*_\varepsilon$ the
corresponding punctured disc.

\begin{definition}
 A second order complex ODE \begin{equation} \label{P0ODE} z''=(p_0+p_1z)z'+(q_3z^3+q_2z^2+q_1z+q_0),\end{equation}
 where the functions $p_i(w),q_j(w)$ are meromorphic in a domain $\Omega\subset\CC{}$, is called  a {\em $\mathcal P_0$-ODE},  if the
 meromorphic coefficients satisfy
  \begin{equation} \label{P0ODErelations}q_3(w)=-\frac{1}{9}p_1^2(w),\quad
  q_2(w)=\frac{1}{3}(p'_1-p_0p_1).\end{equation}
 \end{definition}

In the special case when $\Omega$ is a disc $\Delta_\varepsilon$
and the coefficients $p_i(w),q_j(w)$ have a unique meromorphic
singularity at the point $w=0$, we call \eqref{P0ODE}  a
{\em $\mathcal P_0$\,-\,ODE with a an isolated meromorphic singularity}.
 A $\mathcal P_0$\,-\,ODE with an isolated meromorphic
singularity
 can be always represented as \begin{equation} \label{admODE} z''=\frac{1}{w^m}(Az+B)z'+\frac{1}{w^{2m}}(Cz^3+Dz^2+Ez+F),\end{equation}
 where $m\geq 1$ is an integer and $A(w),B(w),C(w),D(w),E(w),F(w)$ are holomorphic near the origin coefficients, satisfying the special relations
 \begin{equation} \label{ODErelations}C(w)=-\frac{1}{9}A^2(w),\,D(w)=\frac{1}{3}w^{2m}\left(\frac{A(w)}{w^m}\right)'-\frac{1}{3}A(w)B(w).\end{equation}

 Note that it is possible, by scaling the holomorphic coefficients and the denominators $w^m$ and $w^{2m}$ simultaneously, to change the integer $m$
 without changing an ODE \eqref{admODE}.
  To avoid the uncertainty, we call the smallest possible integer $m\geq 1$ for a fixed ODE \eqref{admODE} its {\em singularity order}.
  It is straightforward to check that the special relations \eqref{P0ODErelations}, applied to a $\mathcal P_0$\,-\,ODE $\mathcal E$, are equivalent to the
  fact that the two Tresse semi-invariants \eqref{invariants}
 vanish identically for $w\in\Omega$. Thus,  \eqref{P0ODErelations} is equivalent to the fact that $\mathcal E$ is locally equivalent to $z''=0$ near each \it regular
 \rm point $(z_0,w_0),\,w_0\in\Omega$.

 The $\mathcal P_0$-notation is caused by the fact that the map, transforming  \eqref{P0ODE} into the   simplest ODE $z''=0$, is in fact linear fractional in $z$
 (see, e.g., the proof of Theorem 3.3 in \cite{nonminimalODE}).
 In his celebrated work \cite{painleve} Painlev\'e classified all
second order complex ODEs, rational in the dependent variable $z$
and its derivative, meromorphic in some domain $\Omega$ in the
independent variable $w$, and having no movable critical points
(ODEs of this type are called ODEs of class $\mathcal P$).
The mapping which  brings an ODE of class $\mathcal P$ to its standard
form in this classification, is locally biholomorphic in
$\CP{1}\times\Omega$ and is linear-fractional in the dependent
variable (see, e.g., \cite{ai} for details). In our case the
standard form is flat ($z''=0$), which motivates the $\mathcal
P_0$ notation.

We also note that for $A(w)=C(w)=D(w)=F(w)\equiv 0$ an ODE
\eqref{admODE} is linear (the latter case was considered in
\cite{divergence}), and its Fuchsianity is equivalent to the fact
that its singularity order equals $1$.

A direct calculation shows that if a germ $z(w)$ of a solution of
\eqref{admODE} is
 invertible in some domain, then the inverse function $w(z)$
 satisfies in the image domain the ODE \begin{equation}\label{inverse}
w''=-\frac{1}{w^m}(Az+B)(w')^2-\frac{1}{w^{2m}}(Cz^3+Dz^2+Ez+F)(w')^3.
\end{equation}
 We call \eqref{inverse} \it the inverse ODE \rm for \eqref{admODE}
(i.e., we interchange the dependent and the independent
variables).

 We next introduce a class of anti-holomorphic 2-parameter families of planar
 complex  curves that potentially can be the family of solutions for a $\mathcal P_0$-ODE with an isolated
 meromorphic singularity and, at the same time, the family of Segre varieties of a real hypersurface in $\CC{2}$.

 \begin{definition}  An {\em $m$-admissible Segre family}  is a
 2-parameter antiholomorphic family of planar holomorphic curves
 in a polydisc $\Delta_\delta\times\Delta_\varepsilon$
 which can be parameterized in the form
 \begin{equation} \label{admissiblefamily} w=\bar\eta
 e^{\pm i\bar\eta^{m-1}\varphi(z,\bar\xi,\bar\eta)},\end{equation} where
 $m\geq 1$ is an integer, $\xi\in\Delta_\delta,\eta\in \Delta_\varepsilon$ are holomorphic parameters, and the function $\varphi(x,y,u)$ is holomorphic
 in the polydisc $\Delta_{\delta}\times\Delta_{\delta}\times\Delta_\varepsilon$ and has there an expansion
 $$
 \varphi(x,y,u)=xy+\sum\limits_{k,l\geq 2}\varphi_{kl}(u)x^ky^l, \ \ \varphi_{kl}(u)\in\mathcal O(\Delta_\varepsilon) .
 $$
 \end{definition}

 To avoid confusion in terminology we will call
 $m$-admissible families of the form
 \[\mathcal S = \left\{ w=\bar\eta e^{\pm i\bar\eta^{m-1}\left(z\bar\xi+\sum_{k\geq 2} \psi_k ( \bar\eta )z^k\bar\xi^k\right) } \right\},\] which were considered in \cite{divergence}, {\em $m$-admissible with
 rotations}.
 Thus an $m$-admissible family has the form
 \begin{equation}\label{phi}
 \mathcal S = \left\{w=\bar\eta e^{\pm i\bar\eta^{m-1}\left(z\bar\xi+\sum_{k,l\geq
 2}\varphi_{kl}(\bar\eta)z^k\bar\xi^l\right)},\ (\xi,\eta)\in \Delta_{\delta}\times\Delta_\varepsilon \right\}.
 \end{equation}

$m$-admissibility of an anti-holomorphic 2-parameter family of planar
complex curves can be checked easily: a family defined
by
 $w=\rho(z,\bar\xi,\bar\eta)$, where $\rho$ is holomorphic in some polydisc $U\subset\CC{3}$, centered
 at the origin, is $m$-admissible if and only if the defining function
 $\rho$ has the expansion
 $\rho(z,\bar\xi,\bar\eta)=\bar\eta\pm i\bar\eta^m
 z\bar\xi+O(\bar\eta^m z^2\bar\xi^2)$.

 For a real-analytic  hypersurface $M\subset\CC{2}$ which is nonminimal at the origin with nonminimality
 order~$m$ and is defined by an equation of the form
 \begin{equation}\label{admissiblehyper}
 v=u^m\left(\pm |z|^2+\sum\limits_{k,l\geq 2}h_{kl}(u)z^k\bar z^l\right),
 \end{equation}
 it is not difficult to check that its Segre family is an $m$-admissible Segre family. We call a real hypersurface of
 the form \eqref{admissiblehyper} an {\em  $m$-admissible nonminimal hypersurface}. Note that in the case of
 $m$-admissible Segre families (respectively, nonminimal hypersurfaces) the integer $m$ is uniquely
 determined by the Segre family (respectively, by the hypersurface). Depending on the sign in the exponent
 $e^{\pm i\bar\eta^{m-1}\varphi(z,\bar\xi,\bar\eta)}$ we say that an $m$-admissible Segre family is
 {\em  positive } or {\em negative},  respectively, and apply these notions for real hypersurfaces. In analogy with the case of real hypersurfaces, we call the holomorphic curve in the family
 \eqref{admissiblefamily},  corresponding to the values
 $\xi=a,\eta=b$ of parameters, \it the Segre variety of a point
 \rm  $p=(a,b)\in\Delta_\delta\times\Delta_\varepsilon$ and denote it by $Q_p$.
 % Ilya: a priori depends on parametrization, should maybe note?
We call the hypersurface
 $$
 X=\{w=0\}\subset\Delta_\delta\times\Delta_\varepsilon
 $$
 the {\em singular locus}  of an $m$-admissible Segre family. As a consequence of  \eqref{admissiblefamily}, we have the equivalences
\[Q_p\cap
X\neq\emptyset \Longleftrightarrow p\in X \Longleftrightarrow
Q_p=X. \] Also note that the fact that $w(0)=\bar\eta,\,w'(0)=\pm
i\bar\xi\bar\eta^m$ shows that the {\em Segre mapping }
$\lambda:\,p\longrightarrow Q_p$ is
 injective in $(\Delta_\delta\times\Delta_\varepsilon)\setminus X$.

We next describe a way to connect admissible Segre families with
$\mathcal P_0$\,-\,ODEs.

 \begin{definition}
 We say that an $m$-admissible Segre family $\mathcal S$  {\em is associated with a $\mathcal P_0$\,-\,ODE $\mathcal
  E$ of singularity order $\leq m$},  if after an appropriate shrinking of the basic neighbourhood $\Delta_\delta\times\Delta_\varepsilon$ of the
 origin all the elements $Q_p\in\mathcal S$ with $p\notin X$, considered as graphs
 $w=w(z)$, satisfy the inverse ODE for $\mathcal E$.
 \end{definition}

Note that we may always substitute the Segre
 varieties into \eqref{inverse}.
 Given an ODE $\mathcal E$, we denote an associated $m$-admissible Segre family by
 $\mathcal S^\pm_m(\mathcal E)$, depending on the sign of the Segre family.

 \begin{proposition}\label{pro:odetoseg} For any integer $m\geq 1$ and any $\mathcal P_0$\,-\,ODE $\mathcal
 E$ of singularity order $\leq m$, as in \eqref{admODE}, there
 is a unique positive and a unique negative $m$-admissible Segre family
 $\mathcal S$, associated with $\mathcal E$. The ODE $\mathcal E$ and the associated
 Segre families $\mathcal S^\pm_m(\mathcal E)$ given by \eqref{phi}, satisfy the following relations:
 \begin{gather}
\notag F(w)=2\varphi_{23}(w),\,A(w)=\pm 6i\varphi_{32}(w),\,B(w)=\pm 2i\varphi_{22}(w)-w^{m-1},\\
\label{ODEviadef} E(w)=6\varphi_{33} \pm
2i(m-1)\varphi_{22}w^{m-1}-8(\varphi_{22})^2 \mp
2i\varphi'_{22}w^m. \end{gather}
 In particular, for any fixed $m$ the correspondences $\mathcal E\longrightarrow\mathcal S^+_m(\mathcal E)$
 and $\mathcal E\longrightarrow\mathcal S^-_m(\mathcal E)$ are injective.
 \end{proposition}

 \begin{proof}
Consider a positive $m$-admissible Segre family $\mathcal S$, as
in \eqref{admissiblefamily}, and a $\mathcal P_0$-ODE with an
isolated
 meromorphic singularity $\mathcal
E$. We first express the condition that $\mathcal S$ is associated
with $\mathcal E$ in the form of a differential equation. Fix
$p=(\xi,\eta)\in\Delta_\delta\times\Delta_\varepsilon$ and
consider the Segre variety $Q_p$, given by
\eqref{admissiblefamily}, as a graph $w=w(z)$. For the function
$\varphi(x,y,u)$ we denote by $\dot\varphi$ and $\ddot\varphi$ its
first and second derivatives respectively with respect to the first
argument. Then one computes
\begin{eqnarray*}
w' &=& i\bar\eta^m e^{i\bar\eta^{m-1}\varphi(z,\bar\xi,\bar\eta)}\dot\varphi(z,\bar\xi,\bar\eta),\\
w'' &=& i\bar\eta^m
e^{i\bar\eta^{m-1}\varphi(z,\bar\xi,\bar\eta)}\ddot\varphi(z,\bar\xi,\bar\eta)-
\bar\eta^{2m-1}e^{i\bar\eta^{m-1}\varphi(z,\bar\xi,\bar\eta)}(\dot\varphi(z,\bar\xi,\bar\eta))^2.
\end{eqnarray*}
Plugging these expressions into \eqref{inverse} yields after
simplifications
\begin{gather}\label{findphi}
\ddot\varphi=-i(\dot\varphi)^2\left(\bar\eta^{m-1}+
(A(\bar\eta e^{i\bar\eta^{m-1}\varphi})z+B(\bar\eta e^{i\bar\eta^{m-1}\varphi}))e^{i(1-m)\bar\eta^{m-1}\varphi}\right)+\\
\notag +(\dot\varphi)^3\left(C(\bar\eta
e^{i\bar\eta^{m-1}\varphi})z^3+D(\bar\eta
e^{i\bar\eta^{m-1}\varphi})z^2+E(\bar\eta
e^{i\bar\eta^{m-1}\varphi})z+F(\bar\eta
e^{i\bar\eta^{m-1}\varphi})\right)e^{i(2-2m)\bar\eta^{m-1}\varphi},
\end{gather} where $\varphi=\varphi(z,\bar\xi,\bar\eta)$. The
differential equation \eqref{findphi} is a second order
holomorphic ODE, depending holomorphically on the parameters
$\bar\xi,\bar\eta$. Considering now the Cauchy problem for the ODE
\eqref{findphi} with the initial data
$\varphi(0)=0,\,\dot\varphi(0)=\bar\xi$, we get from the theorem
on the analytic dependence of solutions of a holomorphic ODE on
holomorphic parameters (see, e.g., \cite{ilyashenko}) that its
solution $\varphi=\varphi(z,\bar\xi,\bar\eta)$ is unique and
holomorphic in $z, \bar\xi,\bar\eta$ in some polydisc
$U\subset\CC{3}$, centered at the origin.  Observe that the above
arguments are reversible.

For the proof of the proposition, given a $\mathcal P_0$\,-\,ODE
$\mathcal E$ of singularity order $\leq m$, we solve the
corresponding equation \eqref{findphi} with the initial data
$\varphi(0)=0,\,\dot\varphi(0)=\bar\xi$, and obtain a solution
$\varphi=\varphi(z,\bar\xi,\bar\eta)$. Since
$\varphi(0,\bar\xi,\bar\eta)\equiv
0,\,\varphi_z(0,\bar\xi,\bar\eta)\equiv \bar\xi$, we conclude that
\begin{equation}\label{tempphi}\varphi(z,\bar\xi,\bar\eta)=z\bar\xi+\sum\limits_{k\geq
2,l\geq 0}\varphi_{kl}(\bar\eta)z^k\bar\xi^l.\end{equation}
However, substituting \eqref{tempphi} into \eqref{findphi} and
gathering terms of the form $z^{k-2}\bar\xi^0$ with $k\geq 2$
yields first $\varphi_{20}\equiv 0$ and then by induction
$\varphi_{k0}\equiv 0$ for all $k\geq 2$. Using the latter fact
and gathering in \eqref{findphi} terms of the form
$z^{k-2}\bar\xi^1$ with $k\geq 2$, we get  (since, after the
substitution of \eqref{findphi}, the right hand side in
\eqref{findphi} becomes divisible by $\bar\xi^2$) that
$\varphi_{k1}\equiv 0$ for all $k\geq 2$. Thus $\varphi$ has the
form required for the $m$-admissibility and
$$w=\bar\eta e^{i\bar\eta^{m-1}\varphi(z,\bar\xi,\bar\eta)}$$ is the
desired positive
 $m$-admissible Segre family $\mathcal S=\mathcal
S^+_m(\mathcal E)$ associated with $\mathcal E$. The uniqueness of
$\mathcal
S^+_m(\mathcal E) $ also follows from the uniqueness of the
solution of the Cauchy problem.

To prove the relations \eqref{ODEviadef}, we substitute
\eqref{admissiblefamily} into \eqref{inverse}. We rewrite both
sides of this identity as power series in $z$ and $\bar \bar\xi$
with coefficients depending on $\bar \eta$. If we equate the
coefficients of $\bar \xi^3$, we obtain $2\varphi_{23}(\bar \eta)
= F(\bar \eta)$. Equating terms of the form $z \bar \xi^2$ we
obtain $6i \varphi_{32} (\bar\eta) = A(\bar \eta)$. Similar
computations for $\bar \xi^2$ and $z^3 \bar \xi$ give the formulas
for $B$ and $E$. Finally, to prove the injectiveness one needs to
use, in addition to \eqref{ODEviadef}, the special relations
\eqref{ODErelations}, and this enables to express the whole
$\mathcal E$ in terms of $\mathcal S$. This proves the proposition
in the positive case. The proof in the negative case is analogous.
\end{proof}

 \autoref{pro:odetoseg} gives an effective algorithm for computing the
 $m$-admissible Segre family for a given $\mathcal P_0$-ODE with an isolated
 meromorphic singularity. Our goal is, however, to identify those
 ODEs that produce Segre families with a reality
 condition, that is, Segre families of nonminimal real
 hypersurfaces.

 \begin{definition} We say that an $m$-admissible Segre family  has
 a {\em real structure} if it is the Segre family of an
 $m$-admissible real hypersurface $M\subset \CC{2}$. We also say
 that a $\mathcal P_0$-ODE $\mathcal E$ with an isolated meromorphic singularity of nonsingularity order at most $m$ has {\em $m$-positive
 (respectively, $m$-negative) real structure},
 if the associated positive (respectively, negative) $m$-admissible Segre family $\mathcal S^\pm_m(\mathcal E)$
 has a real structure. We say that the corresponding real hypersurface $M$ is {\em  associated } with $\mathcal E$.
 \end{definition}

 We then need a development, in singular settings, of the concepts of the dual family and dual ODE, described in Section 2.3.
  Let $\rho(z,y,u)$ be a holomorphic function near the origin in $\CC{3}$ with
  $\rho(0,0,0)=0$, and $d\rho(0,0,0)=du$. For  $z,\xi\in\Delta_\delta,w,\eta\in\Delta_\varepsilon$, let
 $$
 \mathcal S=\{w=\rho(z,\bar \xi,\bar \eta)\}
 $$
be a 2-parameter antiholomorphic family of holomorphic curves near
the origin, parametrized by $(\xi,\eta)$. An {\em admissible} parametrization
of  $\mathcal{S}$
is given by a function $\tilde \rho (z, \bar{\xi'}, \bar{\eta'})$ such that
\[ \mathcal{S} = \{ w = \tilde\rho(z,\bar{\xi'}, \bar{\eta'} ) \}  \]
and there exists a  germ of a biholomorphism $(\xi,\eta)\mapsto (\xi',\eta')$
such that $\rho(z,\bar \xi,\bar \eta) = \tilde \rho \left(z,\overline{\xi'(\xi,\eta)},\overline{\eta'(\xi,\eta)} \right)$. Fixing a parametrization and considering all  admissible parametrizations
gives rise to the notion of   a \emph{general Segre family}.

 For each point
 $p=(\xi,\eta)\in\Delta_\delta\times\Delta_\varepsilon$ we call
 the corresponding holomorphic curve
 $Q_p^\rho =\{w=\rho(z,\bar \xi,\bar \eta)\}\in \mathcal S$
its \it  Segre variety. \rm Clearly, an $m$-admissible Segre
family is a particular example of a general Segre family. Note that the
Segre varieties of a general Segre family do depend on the parametrization,
but admissible parametrizations give rise to a relabeling of the Segre varieties
which is ``analytic''.

We say that two general Segre families $\mathcal{S} $ and
 $\tilde {\mathcal{ S}}$ are equivalent if there exists
a germ of a biholomorphism $H=(f,g)$ of $(\CC{2},0)$ such that
$ \tilde {\mathcal{S}} = H^{-1} (\mathcal{S}) $, and such that
the solution of the implicit function problem
$g(z,w) = \tilde \rho( f(z,w)), \xi,\eta)$ for $w$ is an admissible
parametrization of $\mathcal{S}$.

% We say that two (general) Segre families {\em coincide}, if
% there exists a nonempty open neighbourhood $G$ of the origin such
% that for any point $p\in G$ the Segre varieties of $p$ in both
% families coincide.

Further, given a (general) Segre family
$\mathcal S$, from the implicit function theorem one concludes
that the antiholomorphic family of planar holomorphic curves
$$\mathcal
 S^{*,\rho}=\{\bar\eta=\rho(\bar\xi,z,w)\}$$ is also a general Segre family for
 some, possibly, smaller polydisc $\Delta_{\tilde\delta} \times \Delta_{\tilde\varepsilon}$, which depends on the chosen parametrization $\rho$.
 We note that for every admissible parametrization of $\mathcal{S}$, we obtain
 an equivalent Segre family.

 \begin{definition}
 The Segre family $\mathcal S^{*,\rho}$ is called the
 \emph{ dual Segre family } for $\mathcal S$ with the parametrization $\rho$.
 % We will
 % denote by $\mathcal{S}^*$ an arbitrary member of the equivalence
 % class $ \mathcal{S}^{*,\rho} $ where $\rho$ ranges over the admissible
 % parametrizations of $\mathcal{S}$.
 \end{definition}

 The dual Segre family has a simple interpretation:
 in the defining equation of the family $S$ one should consider the
 parameters $\bar\xi,\bar\eta$ as new coordinates, and the variables $z,w$ as new parameters.
 If we denote the Segre variety of a point $p$ with respect to the family
  $\mathcal S^{*,\rho}$ by $Q^{*,\rho}_p$, this just means that
  $Q_p^{*,\rho} = \{ (z,w)\colon \bar p\in Q_{(\bar  z, \bar w)}^\rho  \}$.
  In the following,
  we will suppress the dependence on $\rho$ from the notation whenever
  we make claims which hold for all admissible parametrizations of a given Segre family.

It is not difficult to see that if $\mathcal S$ is a positive
(respectively, negative) $m$-admissible Segre family, then
$\mathcal S^*$ is a negative (respectively, positive)
$m$-admissible Segre family. Indeed, to obtain the defining
function $\rho^*(z,\bar\xi,\bar\eta)$ of the general Segre family
$\mathcal S^*$ we need to solve for $w$  in the equation
\begin{equation}\label{dual}
\bar\eta=w e^{\pm iw^{m-1}\left(z\bar\xi+\sum_{k,l\geq
 2}\varphi_{kl}(w)z^k\bar\xi^l\right)} .
\end{equation}
Note that \eqref{dual} implies
\begin{equation}\label{dual2}
w=\bar\eta e^{\mp iw^{m-1}(z\bar\xi+O(z^2\bar\xi^2))}.
\end{equation}
We then obtain from \eqref{dual2}
$w=\rho^*(z,\bar\xi,\bar\eta)=\bar\eta(1+O(z\bar\xi))$.
Substituting the latter representation into \eqref{dual2} gives
$w=\rho^*(z,\bar\xi,\bar\eta)=\bar\eta e^{\mp
i\bar\eta^{m-1}(z\bar\xi+O(z^2\bar\xi^2))}$, as required.

We also need the following Segre family, connected with $S$:
$$\bar{\mathcal S}=\{w=\bar\rho(z,\bar \xi,\bar \eta)\},$$
where for a power
series of the form
$$f(x)=\sum\nolimits_{\alpha \in\mathbb{Z}^d}c_\alpha x^\alpha
$$  we denote by $\bar f(x)$ the series
$\sum_{\alpha \in\mathbb{Z}^d} \bar c_\alpha x^\alpha $. Note that
$\bar{ \mathcal S}$ does not depend on the particular admissible parametrization, in
contrast to the dual family.

 \begin{definition} The Segre family $\bar{\mathcal S}$ is called  the
 \emph{conjugated family} of $\mathcal S$.
 \end{definition}

 If $\sigma: \CC{2}\longrightarrow\CC{2}$ is the antiholomorphic involution $(z,w)\longrightarrow(\bar
 z,\bar w)$, then one simply has $\sigma(Q_p) = \overline {Q_{\sigma(p)}}$.
 We  will denote the Segre variety of a point $p$ with respect to the family $\bar{\mathcal S}$ by
 $\bar Q_p^\rho $. It follows from the definition that if $\mathcal S$ is a
positive (respectively, negative) $m$-admissible Segre family,
then $\bar{\mathcal S}$ is a negative (respectively, positive)
$m$-admissible Segre family.

 In the same manner as for the case of an $m$-admissible Segre
 family, we say that a (general) Segre family $\mathcal S=\{w=\rho(z,\bar \xi,\bar
 \eta)\}$ has a \it real structure, \rm if there exists a smooth real-analytic
 hypersurface $M\subset\CC{2}$, passing through the origin, such
 that $\mathcal S$ is the Segre family of $M$.

 The use of the  dual and the conjugated Segre families is
 illuminated by the fact that

\medskip

%TODO define Fact environment

 \it A (general) Segre family $\mathcal S$ has a real
 structure if and only if the  conjugated Segre family $\bar{\mathcal S}$ is also
 a dual family, i.e. if there exists an admissible
 parametrization $\rho$  such that $\mathcal S^{*,\rho}=\bar{\mathcal
 S}$ \rm

\medskip

This fact proved, for example, in \cite{divergence} (see Proposition 3.10 there) is a corollary of the
reality condition \eqref{reality} for a real-analytic hypersurface.
Our goal is to transfer the above real structure criterion from
$m$-admissible families to the associated ODEs. In this case, we
can somewhat simplify matters with regard to different parametrizations:
 when working with  admissible families, we will always use the (unique) parametrization
$\rho$  which satisfies the conditions in \eqref{phi} for our constructions.

\begin{definition}
Let $\mathcal E$ be a $\mathcal P_0$-ODE with an isolated
meromorphic singularity of order $\leq m$. We say that a $\mathcal
P_0$-ODE $\mathcal E^*$ with an isolated meromorphic singularity
of order $\leq m$  is  \emph{$m$-dual} to $\mathcal E$, if the
negative $m$-admissible Segre family is dual to the family $\mathcal
S_{\mathcal E}^+$ is associated with $\mathcal E^*$, i.e.,
$$
\mathcal E^* \text{ is } m\text{-dual to } \mathcal E \
\Longleftrightarrow \ (\mathcal S^+_m(\mathcal E))^* =
S^-_m(\mathcal E^*).
$$
In the same manner, we say that a $\mathcal P_0$-ODE
$\overline{\mathcal E}$ with an isolated meromorphic singularity
of order $\leq m$ is  \emph{$m$-conjugated to $\mathcal E$,} if
the negative $m$-admissible Segre family conjugated to the family
$\mathcal S^+_m(\mathcal E)$ is associated with
$\overline{\mathcal E}$, i.e.,
$$
\bar{\mathcal E} \text{  is } m\text{-conjugated  to } \mathcal E
\ \Longleftrightarrow \ \overline{\mathcal S}^+_m(\mathcal E) =
S^-_m(\bar{\mathcal E}).
$$
\end{definition}

From \autoref{pro:odetoseg} we conclude that for a fixed integer $m$ not
preceding the order of a given ODE $\mathcal E$ both the
conjugated and the dual ODEs are unique (if they exist). The existence
of the conjugated ODE for any $m$ as above is obvious: if
$\mathcal E$ is given by
$z''=\frac{1}{w^m}(Az+B)z'+\frac{1}{w^{2m}}(Cz^3+Dz^2+Ez+F),$
then, clearly, the desired ODE $\overline{\mathcal E}$  is given
explicitly by
\begin{equation}\label{conjugatedODE}
z''=\frac{1}{w^m}(\bar Az+\bar B)z'+\frac{1}{w^{2m}}(\bar
Cz^3+\bar Dz^2+\bar Ez+\bar F),\end{equation} where $\bar A=\bar
A(w)$ and similarly for the other coefficients of $\mathcal E$. In
particular, the conjugated ODE does not depend on $m$ and we skip
this parameter for the conjugated ODE in what follows. The
existence of the dual ODE is a much more delicate issue, which uses the triviality of Tresse
semi-invariants of $\mathcal P_0$\,-\,ODEs in a significant way.

\begin{proposition}\label{pro:dualexists} For any $\mathcal P_0$-ODE $\mathcal E$ with an isolated meromorphic singularity of order $\leq m$ the $m$-dual ODE always exists.\end{proposition}

\begin{proof} Suppose first that $\mathcal S$ is positive. Consider the family $\mathcal T=\mathcal (S^+_m(\mathcal E))^*$ and denote
by $\Delta_\delta\times\Delta_\varepsilon$ the polydisc where
$\mathcal T$ is defined.   Take then an arbitrary $p=(\xi,\eta)\in
\Delta_\delta\times\Delta^*_\varepsilon$ and consider Segre
varieties $Q^*_p$ of $\mathcal T$ as graphs $w=w(z)=\bar\eta
e^{-i\bar\eta^{m-1}(z\bar\xi+O(z^2\bar\xi^2))}.$ Then we have
\begin{equation}\label{approximation}
w=\bar\eta+O(z\bar\xi\bar\eta^m),\quad \frac{w'}{w^m}=-i\bar\xi+O(z\bar\xi),
\quad w''=O(\bar\xi^2\bar\eta^m).
\end{equation}
and use the relations \eqref{approximation} in order to obtain a
second order ODE satisfied by all $Q^*_p,p\in
\Delta_\delta\times\Delta^*_\varepsilon$. An application of  the
implicit function theorem to the first two equations in   \eqref{approximation}
yields  functions
$\Lambda(z,w,\zeta)=i\zeta+O(z\zeta)$ and $\Omega(z,w,\zeta)=w+O(zw\zeta)$, such that
$$\bar\xi=\Lambda\left(z,w,\frac{w'}{w^m}\right),\,\bar\eta=\Omega\left(z,w,\frac{w'}{w^m}\right).$$

 Substituting
$\bar\xi=\Lambda(z,w,\frac{w'}{w^m}),\,\bar\eta=\Omega(z,w,\frac{w'}{w^m})$
into the equation for $w''$ in \eqref{approximation} gives us a
second order ODE
\begin{equation}\label{dualODE}
w''=\Phi\left(z,w,\frac{w'}{w^m}\right)\end{equation} for some
function $\Phi(z,w,\zeta)$, holomorphic in a polydisc $\tilde
V\subset\CC{3}$, centered at the origin (compare this with the
elimination procedure in Section~2.2). The ODE \eqref{dualODE} is
satisfied by all $Q^*_p$ with $p\in
\Delta_\delta\times\Delta^*_\varepsilon$. The function
$\Phi(z,w,\zeta)$ also satisfies $\Phi(z,w,\zeta)=O(\zeta^2w^m)$.

On the other hand, the  holomorphic 2-parameter family $\mathcal
S$ can be locally biholomorphically mapped into the family of
affine straight lines in $\CC{2}$ near each regular point of it,
i.e., near each point with $w\neq 0$ (see the discussion in the
beginning of the section). According to Section 2, the same
property holds for the dual family $\mathcal T$. In particular,
Tresse's semi-invariants \eqref{invariants} vanish identically for
the ODE \eqref{dualODE}, and hence $\frac{\partial^4}{(\partial
w')^4}\left[\Phi\left(z,w,\frac{w'}{w^m}\right)\right]\equiv 0$.
The latter means that the function $\Phi(z,w,\zeta)$ is at most cubic
in its third argument. Since, in addition,
$\Phi(z,w,\zeta)=O(\zeta^2w^m)$, we conclude  that we can write
 \[ \Phi(z,w,\zeta) = w^m(\Phi_2(z,w)\zeta^2+\Phi_3(z,w)\zeta^3 )\]
 for some functions
$\Phi_2(z,w)$ and $\Phi_3(z,w)$ holomorphic in a polydisc
$\Delta_r\times\Delta_R$. Then the substitution
$\zeta=\frac{w'}{w^m}$ turns \eqref{dualODE} into an ODE
\begin{equation}\label{dualODE1}
w''=\frac{\Phi_2(z,w)}{w^m}(w')^2+\frac{\Phi_3(z,w)}{w^{2m}}(w')^3.\end{equation}

We claim that the functions $\Phi_2(z,w)$ and $\Phi_3(z,w)$ in
\eqref{dualODE1} are actually polynomials in $z$ of degree $1$ and $3$
respectively. Let
$(z_0,w_0)\in\Delta_\delta\times\Delta^*_\varepsilon$ and choose a
small enough polydisk  $U$ centered at $(z_0,w_0)$ such that
there exists a locally biholomorphic
mapping $\mathcal F:\,Z=f(z,w),\,W=g(z,w)$ of the polydisc $U$
into $\CC{2}$, transforming \eqref{dualODE1} into the
ODE\, $W''=0$. Performing a recalculation of first and second
order derivatives in the coordinates $(z,w)$ (see Section 2.3) we
get that \eqref{dualODE1} is given in $U$ by
\begin{equation}\label{dualODE2}
w''=I_0(z,w)+I_1(z,w)w'+I_2(z,w)(w')^2+I_3(z,w)(w')^3,
\end{equation}
where

\begin{equation}\label{classical}
\begin{aligned}
    &   I_0 =\frac{1}{f_wg_z-g_wf_z}\left(f_zg_{zz}-g_zf_{zz}\right) ,\\
& I_1 =\frac{1}{f_wg_z-g_wf_z}\left(f_wg_{zz}-g_wf_{zz}+2f_zg_{zw}-2g_zf_{zw}\right) ,\\
&  I_2 =\frac{1}{f_wg_z-g_wf_z}\left(f_zg_{ww}-g_zf_{ww}+2f_wg_{zw}-2g_wf_{zw}\right) ,\\
&  I_3
=\frac{1}{f_wg_z-g_wf_z}\left(f_wg_{ww}-g_wf_{ww}\right) .
\end{aligned}\end{equation}
(since $\mathcal F$ is biholomorphic in $U$, the Jacobian
$J=f_wg_z-f_zg_w$ is nonzero in $U$). Comparing \eqref{dualODE1}
and \eqref{dualODE2} we conclude that the two functions
$I_0(z,w),I_1(z,w)$ vanish identically in $U$ and that
 $\frac{\Phi_2(z,w)}{w^m}=I_2(z,w),\,\frac{\Phi_3(z,w)}{w^{2m}}=I_3(z,w)$.
In particular, we have that $(f,g)$ satisfies the PDE system
 \begin{equation}\label{PDEsystem}
\begin{aligned}
 f_zg_{zz}-g_zf_{zz}&=0\\
f_wg_{zz}-g_wf_{zz}+2f_zg_{zw}-2g_zf_{zw}&\tabularnewline
=0.
\end{aligned}
 \end{equation}

As was shown in \cite{nonminimalODE} (see the proof of Theorem 3.3
there), any solution $(f,g)$ of the system \eqref{PDEsystem} with
$J(z,w)\neq 0$ is linear-fractional in $z$ in the polydisc $U$,
i.e., there exists six holomorphic in $U$ functions
$\alpha_j(w),\beta_j(w),\,j=0,1,2$ such that
$$f=\frac{\alpha_1(w)z+\beta_1(w)}{\alpha_0(w)z+\beta_0(w)},\quad g=\frac{\alpha_2(w)z+\beta_2(w)}{\alpha_0(w)z+\beta_0(w)}.$$
After
 composing $\mathcal F$ with an appropriate element
$\sigma\in\mbox{Aut}(\CP{2})$ (this group preserves the
target ODE $W''=0$) if needed,  we can assume without
loss of generality that
$\alpha_0(w)\not\equiv 0$,
and rewrite $f$ and $g$ as
\begin{equation}\label{fgaalpha}
f(z,w)=\frac{\alpha}{z+\delta}+\beta,\quad g(z,w)=\frac{a}{z+\delta}+b,
\end{equation}
for appropriate $\alpha(w),\beta(w),\delta(w),a(w),b(w)$,
meromorphic near $w_0$; one checks that these need to satisfy $ \alpha b' = \beta a' $
if \eqref{PDEsystem} is satisfied.

If we now substitute the expressions \eqref{fgaalpha} of $f$ and $g$
 into $I_2(z,w)$ and
$I_3(z,w)$, we obtain affine-linear and cubic
expressions in $z$,  respectively, more precisely, we have
\begin{gather*} I_2(z,w)=\left[\frac{a\alpha''-\alpha
a''}{a'\alpha-\alpha'a}\right]+3\left[\frac{b'\alpha'-\beta'a'}{a'\alpha-\alpha'a}\right](z+\delta),\\ %note:coefficient of z+\delta can also be written as
% - 3 \frac{\beta'}{\alpha}
 I_3(z,w)=\left[\delta''+\delta'\frac{a\alpha''-\alpha
a''}{a'\alpha-\alpha'a}\right]+\left[\frac{a''\alpha'-\alpha''
a'}{a'\alpha-\alpha'a}+3\delta'\frac{b'\alpha'-\beta'a'}{a'\alpha-\alpha'a}\right](z+\delta)+\\
\notag
+\left[\frac{\beta'a''-b'\alpha''+\alpha'b''-a'\beta''}{a'\alpha-\alpha'a}\right](z+\delta)^2+
\left[\frac{\beta'b''-b'\beta''}{a'\alpha-\alpha'a}\right](z+\delta)^3.
%note: coefficient of (z+\delta)^3 can also be written as \left( \frac{\beta'}{\alpha'} \right)^2
\end{gather*}
This implies the desired polynomial dependence of $\Phi_2(z,w)$
and $\Phi_3(z,w)$ on $z$.

Clearly, the obtained property of $\Phi_2,\Phi_3$ is equivalent to
the fact that \eqref{dualODE1} has the form \eqref{admODE}. Since
\eqref{dualODE1} is mappable into the simplest ODE $w''=0$ near
its regular points (see the arguments above), its Tresse
semi-invariants vanish identically, which yields the special
relations \eqref{ODErelations}. Thus \eqref{dualODE1} is a
$\mathcal P_0$-ODE with an isolated meromorphic singularity of
order $\leq m$, which proves the proposition.
\end{proof}

We immediately get the following criterion for identifying ODEs with
a real structure.

\begin{corollary}\label{cor:realstruc} A $\mathcal P_0$-ODE with an isolated meromorphic
singularity of order $\leq m$ has an $m$-positive real structure
if and only if its $m$-dual ODE
 coincides with the conjugated one: $\mathcal E^*_m=\bar{\mathcal
 E}$.\end{corollary}

Before providing the real structure criterion for $\mathcal
P_0$-ODEs we need a computational

\begin{lemma}\label{lem:segrelations} Let $\mathcal S$ be a positive $m$-admissible Segre family, and $$\mathcal S=\left\{w=\bar\eta
e^{i\bar\eta^{m-1}\varphi}\right\},\quad \bar{\mathcal
S}=\left\{w=\bar\eta
e^{-i\bar\eta^{m-1}\tilde\varphi}\right\},\quad \mathcal{ S}^*=\left\{w=\bar\eta e^{-i\bar\eta^{m-1}\varphi^*}\right\}.$$ %NOTE: Was S^*_m, think you meant this?
Then
\begin{gather} \label{phibar}
\tilde\varphi_{kl}(w)=\bar\varphi_{kl}(w),\,k,l\geq 2\\
\label{phistar22}\varphi^*_{22}(w)=\varphi_{22}(w)-i(m-1)w^{m-1},\quad
\varphi^*_{32}(w)=\varphi_{23}(w),\quad \varphi^*_{23}(w)=\varphi_{32}(w),\\
\label{phistar33}
\varphi^*_{33}=\varphi_{33}(w)+\frac{3}{2}(m-1)^2w^{2m-2}-2i(m-1)w^{m-1}\varphi_{22}(w)-iw^m\varphi'_{22}(w). %NOTE: Equations double-checked.
\end{gather}\end{lemma}

\begin{proof}
The relations \eqref{phibar} follow directly from the definition
of $\bar{\mathcal S}$. To prove
\eqref{phistar22},\eqref{phistar33} we write $\mathcal S^*_m$ up
first by  definition as
$$\bar\eta=w\exp\left[{iw^{m-1}\left(z\bar\xi+\sum_{k,l\geq
2}\varphi_{kl}(w)z^l\bar\xi^k\right)}\right],$$ and then as
$$w=\bar\eta\exp\left[{-i\bar\eta^{m-1}\left(z\bar\xi+\sum_{k,l\geq
2}\varphi^*_{kl}(\bar\eta)z^k\bar\xi^l\right)}\right].$$
Substituting the first representation into the second and
simplifying, we get
\begin{gather}\notag
\exp\left[{i(m-1)w^{m-1}\left(z\bar\xi+\sum\nolimits_{k,l\geq
2}\varphi_{kl}(w)z^l\bar\xi^k\right)}\right]\times\\
\label{longrelation}\times\left(z\bar\xi+\sum\nolimits_{k,l\geq
2}\varphi^*_{kl}(\bar\eta)z^k\bar\xi^l\right)\left|_{\bar\eta=we^{iw^{m-1}\varphi(\bar\xi,z,w)}}\right.
=z\bar\xi+\sum\nolimits_{k,l\geq
2}\varphi_{kl}(w)z^l\bar\xi^k.\end{gather}

 Gathering the terms with $z^2\bar\xi^2,
z^3\bar\xi^2,z^2\bar\xi^3$ respectively in
\eqref{longrelation},  we get  the
first, the second and the third identities in \eqref{phistar22}.
Gathering then terms with $z^3\bar\xi^3$ and using
\eqref{phistar22}, we obtain \eqref{phistar33}, which proves the
lemma.
\end{proof}

We are in the position now to prove the main result of this
section.

\begin{theorem}\label{thm:classify}
Let
\[\mathcal
E:\,z''=\frac{1}{w^m}(Az+B)z'+\frac{1}{w^{2m}}(Cz^3+Dz^2+Ez+F)\]
be
a $\mathcal P_0$-ODE with an isolated meromorphic singularity of
order $\leq m$, $w\in\Delta_r,r>0,\,m\in\mathbb{N}$. Then
$\mathcal E$ has an $m$-positive real structure if and only if the
functions $A(w),B(w),C(w),D(w),E(w),F(w)$ are given by
% \begin{eqnarray}\notag
% B(w)&=&2ia(w)-mw^{m-1},\ E(w)=b(w)+iw^ma'(w),\\
% \label{keyformulas} F(w)&=&c(w)+id(w),\,A(w)=3d(w)+3ic(w),\,
% C(w)=-(c(w)+id(w))^2,\\
% \notag
% D(w)&=&w^{2m}\left(\frac{d(w)+ic(w)}{w^{m}}\right)'-(2ia(w)-mw^{m-1})(d(w)+ic(w)),\end{eqnarray}
\begin{equation}\label{keyformulas}
\begin{aligned}
A(w)&=3c(w), \\
B(w)&=2ia(w)-mw^{m-1}, \\
C(w)&=\bar c (w) ^2,\\
D(w)&=w^m c'(w) - 2 i a(w) c(w) \\
 E(w)&=b(w)+iw^ma'(w),\\
 F(w)&=i\bar c(w)
\end{aligned}
\end{equation}
% \begin{equation}\label{keyformulas}
% \begin{aligned}
% A(w)&=3d(w)+3ic(w), \\
% B(w)&=2ia(w)-mw^{m-1}, \\
% C(w)&=-(c(w)+id(w))^2,\\
% D(w)&=w^{2m}\left(\frac{d(w)+ic(w)}{w^{m}}\right)'-(2ia(w)-mw^{m-1})(d(w)+ic(w)),\\
%  E(w)&=b(w)+iw^ma'(w),\\
%  F(w)&=c(w)+id(w)
% \end{aligned}
% \end{equation}

for some  power series
\[a(w)=\sum\limits_{j=0}^\infty
a_jw^j,\,b(w)=\sum\limits_{j=0}^\infty
b_jw^j \in \RR{}\{ w\} \, \text{, and } c(w)=\sum\limits_{j=0}^\infty
c_jw^j \in \CC{} \{w\}\]
which converge in $\Delta_r$.
Moreover, if $\mathcal E$ has an $m$-positive real
structure, then the associated real hypersurface $M\subset\CC{2}$
is Levi nondegenerate and spherical outside the complex locus
$X=\{w=0\}$.
\end{theorem}

\begin{proof} As previously
observed, the conjugated  ODE $\overline{\mathcal E}$ has the form
$$z''=\frac{1}{w^m}(\bar Az+\bar B)z'+\frac{1}{w^{2m}}(\bar Cz^3+\bar Dz^2+\bar Ez+\bar
F).$$ We write  the dual ODE $\mathcal E^*_m$ as
$$\mathcal
E^*\colon z''=\frac{1}{w^m}(A^*z+B^*)z'+\frac{1}{w^{2m}}(C^*z^3+D^*z^2+E^*z+F^*)$$
and assume that the families $\mathcal S=\mathcal S^+_m(\mathcal E)$, $\mathcal
S^*$, and $\bar{\mathcal S}$ are given in a polydisc
$U=\Delta_\delta\times\Delta_\varepsilon$  by
\[ \mathcal S = \left\{w=\bar\eta
 e^{i\bar\eta^{m-1}\varphi(z,\bar\xi,\bar\eta)}\right\},\quad
 \mathcal
S^* =\left\{w=\bar\eta
 e^{-i\bar\eta^{m-1}\varphi^*(z,\bar\xi,\bar\eta)}\right\},\quad \bar{\mathcal S} =
 \left\{w=\bar\eta
 =  e^{-i\bar\eta^{m-1}\bar\varphi(z,\bar\xi,\bar\eta)}\right\},\]
with $\varphi,\,\varphi^*$ as in
 \eqref{phi}.  According to
 \autoref{cor:realstruc}, $\mathcal E$ has an $m$-positive real structure if an only if
 $$\bar A(w)=A^*(w),\,\bar B(w)=B^*(w),\,\bar C(w)=C^*(w),\,\bar D(w)=D^*(w),\,\bar E(w)=E^*(w),\,\bar F(w)=F^*(w).$$ It follows directly from
 \eqref{ODEviadef} and \eqref{ODErelations} that the latter conditions are equivalent to
 \begin{equation}\label{upto3jet}
 \bar\varphi_{kl}=\varphi^*_{kl},\,k,l\in\{2,3\}.
 \end{equation}

 \autoref{lem:segrelations} now implies that  \eqref{upto3jet} is equivalent to the
 existence of power series $a(w),\tilde b(w)\in\RR{}\{w\}$ and $c(w)\in\CC{}\{w\}$, convergent in some disc $\Delta_r$, such that
%  \begin{equation}\label{keyformulas1}
% \begin{aligned}
%  \varphi_{22}(w)&=a(w)+i\frac{m-1}{2}w^{m-1},\\
%  \varphi_{23}(w)&=\frac{1}{2}(c(w)+id(w)),\\
%  \varphi_{32}(w)&=\frac{1}{2}(c(w)-id(w)),\\
%  \varphi_{33}(w)&= \tilde b(w)+
%  \frac{i}{2}w^ma'(w)+i(m-1)w^{m-1}a(w).
% \end{aligned}
%  \end{equation}
\begin{equation}\label{keyformulas1}
\begin{aligned}
 \varphi_{22}(w)&=a(w)+i\frac{m-1}{2}w^{m-1},\\
 \varphi_{23}(w)&=\frac{i}{2}  \bar c(w),\\
 \varphi_{32}(w)&= - \frac{i}{2}  c (w)\\
 \varphi_{33}(w)&= \tilde b(w)+
 \frac{i}{2}w^ma'(w)+i(m-1)w^{m-1}a(w).
\end{aligned}
 \end{equation}
  Applying
 \eqref{ODEviadef} and \eqref{ODErelations} again, we conclude that with
 \begin{equation}\label{findb} b(w):=6\tilde b(w)-8a^2(w)+2(m-1)^2w^{2m-2} \end{equation}
 \eqref{keyformulas1} is equivalent to \eqref{keyformulas}.

 It remains to prove that if $\mathcal E$ has an $m$-positive real structure,
 then the associated nonminimal real hypersurface $M\subset\CC{2}$
 is Levi nondegenerate and spherical in the complement to the
 singular set
 $X=\{w=0\}$. Recall that the Segre family of $M$ near the origin
 coincides with $\mathcal S$.
  To prove the Levi nondegeneracy of $M$ in $M\setminus X$ we first note that the Segre map $\lambda \colon p\mapsto Q_p$ is one-to-one in $U\setminus X$
  (see the arguments in the beginning of the section). Consider now any two distinct points $p,q\not\in X$ and their Segre varieties
  $Q_p,Q_q$, $Q_p\cap Q_q\ni r$.
  The fact that $Q_p,Q_q$  are two distinct solutions of the \emph{ nonsingular }
  ODE  $\mathcal E$ in
  $U\setminus X$ implies that their intersection at $r\in U\setminus X$ is
 transverse. Accordingly, any Segre variety of $M$ near an arbitrary point $s\in M\setminus X$ is  uniquely determined by its
 1-jet at a given point, and hence  $M$ is Levi nondegenerate at $s$
 (see, e.g.,
 \cite{DiPi},\cite{ber}).

 Finally, to prove that $M$ is spherical at any $s\in M\setminus X$, we note that the Segre family $\mathcal S$ of $M$ satisfies the $\mathcal P_0$-ODE $\mathcal E$
 and hence is locally biholomorphically mappable in a neighborhood $V$ of $s$ onto the family of straight affine lines in $\CC{2}$. It is not difficult to verify from here that
the image of $M\cap V$ under such a mapping is contained in a
quadric $\mathcal Q\subset\CP{2}$ (see, for example, the proof of
Theorem 6.1 in \cite{nonminimal}), which implies sphericity of $M$
at $s$.  \end{proof}

A completely anaologous argument as in
the case of positive Segre families gives a complete characterization of ODEs with a
negative real structure: these are obtained by conjugating ODEs
with a positive real structure. Thus we can formulate

\begin{corollary}
Let
\[ \mathcal
E \colon z''=\frac{1}{w^m}(Az+B)z'+\frac{1}{w^{2m}}(Cz^3+Dz^2+Ez+F)\] be
a $\mathcal P_0$-ODE with an isolated meromorphic singularity of
order $\leq m$, $w\in\Delta_r,r>0,\,m\in\mathbb{N}$. Then
$\mathcal E$ has an $m$-negative real structure if and only if the
the conjugated
ODE
$$\bar{\mathcal E} \colon z''=\frac{1}{w^m}(\bar
Az+\bar B)z'+\frac{1}{w^{2m}}(\bar Cz^3+\bar Dz^2+\bar Ez+\bar
F)$$ satisfies the relations \eqref{keyformulas} for some
 power series $a(w),b(w)\in\RR{}\{w\},\,c(w)\in\CC{}\{w\}$, which converge  in some $\Delta_r$.
Moreover, if $\mathcal E$ has an $m$-negative
real structure, then the associated real hypersurface
$M\subset\CC{2}$ is Levi nondegenerate and spherical outside the
complex locus $X=\{w=0\}$.
\end{corollary}

 \autoref{thm:classify} and the proof of \autoref{pro:odetoseg} enable us
 to complete the study of the class of
 real-analytic nonminimal at the origin real hypersurfaces $M\subset\CC{2}$ which are spherical outside their complex locus $X\ni 0$.
 More precisely, we present an effective algorithm for
obtaining real-analytic hypersurfaces
$M\subset\CC{2}$, nonminimal at the origin,
with prescribed nonminimality order $m\geq 1$,
which are Levi nondegenerate and spherical outside the nonminimal locus
$X\subset M$. Moreover, one can prescribe to $M$ an arbitrary
6-jet, satisfying the reality condition \eqref{reality}. In fact,
the result of \cite{nonminimalODE} shows that for any
hypersurface $M$ as above there exists appropriate local
holomorphic coordinates near the origin and a $\mathcal P_0$-ODE
$\mathcal E$ such that the Segre family of $M$ in these
coordinates is associated with $\mathcal E$, and thus this
algorithm describes \it all \rm possible hypersurfaces of our
class. We summarize this algorithm below.

\begin{algorithm}\label{alg:nonminimal}
\caption{Algorithm for obtaining nonminimal spherical real
hypersurfaces}
Take three
power series $a(w),b(w),c(w),$ where $a(w),b(w)\in\RR{}\{w\},\,c(w)\in\CC{}\{w\}$, which converge in some
disk centered at the origin,  an
integer $m\geq 1$, and compute six functions
$A(w),B(w),C(w),D(w),E(w),F(w)$ by the formulas
\eqref{keyformulas}. This gives a $\mathcal P_0$-ODE
\eqref{admODE} with an isolated meromorphic singularity of order
$\leq m$. \\
Solve the holomorphic ODE \eqref{findphi} with holomorphic
parameters $\bar\xi,\bar\eta$ and the initial data
$\dot\varphi(0)=0,\,\dot\varphi(0)=\bar\xi$ to obtain a
function $\varphi(z,\bar\xi,\bar\eta)$, holomorphic near the origin in $\CC{3}$.
\\
Either of the two equation $w=\bar w e^{i\bar w^{m-1}\varphi(z,\bar
z,\bar w)}$ and $w=\bar w e^{-i\bar w^{m-1}\bar\varphi(z,\bar
z,\bar w)}$ determines a  real-analytic
hypersurface $M^\pm\subset\CC{2}$, nonminimal at the origin of nonminimality
order $m$, Levi nondegenerate and spherical outside the nonminimal
locus $X=\{w=0\}$. The 6-jets of $M^\pm$ in $z$ are determined by
finding $\tilde b(w)$ from \eqref{findb} and then
$\varphi_{22},\varphi_{23},\varphi_{32},\varphi_{33}$ by formulas
\eqref{keyformulas1}.
\end{algorithm}

% \begin{center} \bf Algorithm for obtaining nonminimal spherical real
% hypersurfaces \end{center}

% \begin{enumerate}[1.]
% \item  Take four
% power series $a(w),b(w),c(w),d(w)$ with real coefficients, which converge in some
% disk centered at the origin,  an
% integer $m\geq 1$, and compute six functions
% $A(w),B(w),C(w),D(w),E(w),F(w)$ by the formulas
% \eqref{keyformulas}. This gives a $\mathcal P_0$-ODE
% \eqref{admODE} with an isolated meromorphic singularity of order
% $\leq m$.
% \item Solve the holomorphic ODE \eqref{findphi} with holomorphic
% parameters $\bar\xi,\bar\eta$ and the initial data
% $\dot\varphi(0)=0,\,\dot\varphi(0)=\bar\xi$ to obtain a
% function $\varphi(z,\bar\xi,\bar\eta)$, holomorphic near the origin in $\CC{3}$.
% \item
% Either of the two equation $w=\bar w e^{i\bar w^{m-1}\varphi(z,\bar
% z,\bar w)}$ and $w=\bar w e^{-i\bar w^{m-1}\bar\varphi(z,\bar
% z,\bar w)}$ determines a  real-analytic
% hypersurface $M^\pm\subset\CC{2}$, nonminimal at the origin of nonminimality
% order $m$, Levi nondegenerate and spherical outside the nonminimal
% locus $X=\{w=0\}$. The 6-jets of $M^\pm$ in $z$ are determined by
% finding $\tilde b(w)$ from \eqref{findb} and then
% $\varphi_{22},\varphi_{23},\varphi_{32},\varphi_{33}$ by formulas
% \eqref{keyformulas1}.
% \end{enumerate}

\begin{corollary}\label{cor:algorithm} Algorithm~\ref{alg:nonminimal} gives a complete description, up to a local biholomorphic equivalence, of all possible
 real-analytic nonminimal at the origin real hypersurfaces $M\subset\CC{2}$, spherical outside their complex locus $X\ni 0$.
\end{corollary}

\section{CR-mappings without analyticity}

In this section we provide a construction  of real-analytic
holomorphically nondegenerate real hypersurfaces and $C^\infty$
CR-diffeomorphisms between them which are not everywhere analytic.

The desired real hypersurfaces are associated with singular ODEs
with a real structure as studied in \autoref{sec:nonminimal}.
We make the particular choice
\[ a(w)\equiv 1,\quad b(w)=\gamma w^{4},\quad c(w)\equiv 0, \]
where
$\gamma\in\RR{}$ is a real constant, and apply Algorithm~\ref{alg:nonminimal}.
In the first step, applying  formulas
\eqref{keyformulas} with $m=4$, we obtain a
one-parameter family $\mathcal E_\gamma$ of $\mathcal P_0$-ODEs
(in fact,  \emph{linear}  ODE) with an isolated meromorphic
singularity of order $4$, which have a $4$-positive real
structure:

\begin{equation}\label{mainODE}
\mathcal E_\gamma \colon
z''=\left(\frac{2i}{w^4}-\frac{4}{w}\right)z'+\frac{\gamma}{w^4}z.
\end{equation}
Each ODE $\mathcal E_\gamma$ has a  non-Fuchsian
singularity at the origin. We denote by $M_\gamma$ the
$4$-nonminimal at the origin real hypersurfaces, associated with
$\mathcal E_\gamma$. Each $M_\gamma$ is Levi nondegenerate and
spherical outside the complex locus $X=\{w=0\}$. Note that the ODE
$\mathcal E_0$ coincides with the ODE $\mathcal E^4_0$, studied in
\cite{divergence}, while the ODEs $\mathcal E_\gamma$ with
$\gamma\neq 0$ are different from that in \cite{divergence}.

After introducing $u:=z'w^3$ as a new dependent variable we
rewrite \eqref{mainODE} as the first order system
\begin{equation} \label{mainsystem} \begin{pmatrix} z \\ u
\end{pmatrix}'=\frac{1}{w^4}\left(A_0+A_1w+A_3w^3\right)\begin{pmatrix} z \\ u
\end{pmatrix},\end{equation}
which possesses a non-Fuchsian singularity at the
origin, where $$A_0=\begin{pmatrix} 0 & 0 \\ 0 & 2i
\end{pmatrix},\,A_1=\begin{pmatrix} 0 & 1 \\ 0 & 0
\end{pmatrix},\,A_3=\begin{pmatrix} 0 & 0 \\ \gamma & -1
\end{pmatrix}.$$

We need to consider  three different
kinds of local transformations in the following: holomorphic, formal and sectorial.
To introduce the latter ones, we denote by
$S_\alpha^{\pm}$ the
unbounded sectors
\[ S_{\alpha}^+ =\{-\alpha<\mbox{Arg} w<\alpha\}, \quad
S_{\alpha}^- = \{\pi-\alpha<\mbox{Arg} w<\pi+\alpha\},
 \] where
 $0<\alpha<\frac{\pi}{2}$, and by $S^\pm_{\alpha,r}$
the bounded sectors $S^\pm_\alpha\cap\Delta_r$, where $r>0$.

\begin{definition}
We say that $F(z,w):
\,(\CC{2},0)\lr(\CC{2},0)$ is a \emph{(formal) gauge transformation} if
there exist (formal) power series $f(w)$, $g(w)$ satisfying $f(0)\neq 0$, $g(0)=0$,
 $g'(0)\neq 0$
such that
\[ F(z,w) = (z f(w), g(w)).  \]
A \emph{sectorial gauge transformation} $F(z,w)$ is a holomorphic
map $F:\,\Delta_R\times
S^\pm_{\alpha,r}\to \Delta_{\tilde R}\times S^\pm_{\tilde
\alpha,\tilde r}$ which is of the form $F(z,w) = (z f(w), g(w))$ where
$f(w)$ and $g(w)$ are holomorphic on $S^\pm_{\alpha,r}$, and whose
asymptotic expansion $\hat F$ is a formal gauge transformation.
\end{definition}

We will denote the groups of
 holomorphic or formal  gauge transformations by   $\mathcal G$ and
$\mathcal {FG}$,  respectively, and for any integer $m\geq
2$,
 $\mathcal G_m\subset\mathcal G$ and
$\mathcal{FG}_m\subset\mathcal{FG}$ will denote the subgroups
whose elements $(zf(w), g(w))$  satisfy the
normalization conditions $f(0)=1$, $g(w)=w+O(w^{m+1})$.

% Similarly, for any
% fixed sign $+$ or $-$ we define the group $\mathcal{SG}^\pm$ of
% \it sectorial gauge transformations \rm $F:\,\Delta_R\times
% S^\pm_{\alpha,r}\lr\Delta_{\tilde R}\times S^\pm_{\tilde
% \alpha,\tilde r}$, $F(z,w)=(zf(w),g(w)$, where $f(w),g(w)$ are
% holomorphic in a sector $S^\pm_{\alpha,r}$ and have there an
% asymptotic expansions
% $f=a_0+a_1w+...,\,g=b_1w+b_2w^2+...,\,a_0,b_1\neq 0$. Also
% $\mathcal{SG}^\pm_m\subset\mathcal{SG}$ denotes the subgroup
% normalized by $a_0=b_1=1,\,b_2=...=b_m=0$. Note that
% $r,\alpha,\tilde r,\tilde \alpha,R,\tilde R$ here are not fixed.

One can define, in the natural way, equivalence of
$\mathcal P_0$-ODEs by means of homomorphic, formal or sectorial
gauge transformations.

\begin{proposition}\label{pro:formalequivalODE} For any  $\gamma\in\RR{}$ the ODE $\mathcal
E_\gamma$ is formally equivalent to the ODE $\mathcal E_0$ by
means of a transformation $F\in\mathcal {FG}_4$.\end{proposition}

\begin{proof} The main tool of the proof is the Poincare-Dulac
normalization procedure for nonresonant non-Fuchsian systems (see,
e.g., \cite{ilyashenko},\cite{vazow}). Such a normal form enables
one to find the fundamental system of formal solutions of a
non-Fuchsian system.

It is straightforward to verify that the function
$\exp\left(-\frac{2i}{3}w^{-3}\right)$ is a solution of the ODE
$\mathcal E_0$, so that the fundamental system of solutions for
$\mathcal E_0$ is
$\left\{1,\exp\left(-\frac{2i}{3}w^{-3}\right)\right\}$. For the
ODE $\mathcal E_\gamma$ with $\gamma\neq 0$ we consider the
corresponding system \eqref{mainsystem} and note that the
principal
matrix \[A_0=\begin{pmatrix} 0 & 0 \\
0 & 2i\end{pmatrix}\]
is diagonal with distinct eigenvalues;
hence the system is nonresonant.
When we perform a transformation of the form
\[ \begin{pmatrix} z \\ u
\end{pmatrix}\lr (I+wH)\begin{pmatrix} z \\ u
\end{pmatrix},\]
where $I$ is the identity matrix and $H$ is a constant $2\times 2$ matrix,
we obtain the transformed system
\[\begin{pmatrix} z \\ u
\end{pmatrix}'=\frac{1}{w^4}\tilde A(w)\begin{pmatrix} z \\ u
\end{pmatrix} = \frac{1}{w^4} \left( \tilde A_0+\tilde
A_{1}w\dots \right), \]
 where $\tilde
A(w)=(I+wH)^{-1}\left(A(w)(I+wH)-Hw^4\right)$. By comparing coefficients of
 $w^{k},k\geq -4$, one computes that
\begin{equation*}
 \begin{aligned}
  \tilde A_0&=A_0, \\ \tilde A_1&=[A_0,H]+A_1.%A_0H-HA_0+A_1.
 \end{aligned}
 \end{equation*}
One can choose $H$ so that $\tilde
A_1=0$ by solving the equation $[A_0,H]=-A_1$, which can be done explicitly:
 \[H=\begin{pmatrix} 0 & \frac{i}{2} \\
0 & 0
\end{pmatrix}.\]
Note that $H^2=HA_1=A_1H=0$. We then get
$$\tilde A_2=A_1H-HA_0H-HA_1+H^2A_0=H [H,A_0] =HA_1=0$$ and
$$\tilde A_3=A_3-HA_1H+H^2A_0H+H^2A_1-H^3A_0=A_3.$$ Thus
 $\tilde A(w) = A_0+A_3w^3+O(w^4)$.

A computation, similar to the above one, shows that the
offdiagonal element $-\gamma$ of the matrix $\tilde A_3$ can be
removed by a transformation
 \[\begin{pmatrix} z
\\ u
\end{pmatrix}\lr (I+w^{3}\tilde H)\begin{pmatrix} z \\ u
\end{pmatrix}\]
 for an appropriate $2\times 2$ constant matrix $\tilde
H$ without changing the 2-jet of $\tilde A(w)$ and the diagonal of $\tilde A_3$.
The matrix $\tilde H$ can be found from the
equation $\tilde A_3+[A_0, \tilde H ]=0$, and one
can choose, for example,
\[\tilde H=\frac{1}{2i}\begin{pmatrix} 0 & 1 \\ -\gamma & 0 \end{pmatrix}.\]

Finally, the matrices $\tilde A_k$ with $k\geq 4$ correspond to holomorphic terms
in the expansion of $\frac{1}{w^4}\tilde A(w)$ and hence can be
removed by the Poincare-Dulac formal normalization procedure for
nonresonant non-Fuchsian systems, without changing the $3$-jet of
the matrix $\tilde A(w)$ (see, e.g., \cite{ilyashenko}, Theorem
20.7). Thus the formal normal form of the system
\eqref{mainsystem} becomes
\begin{equation}\label{normalform}
\begin{pmatrix} z \\ u
\end{pmatrix}'=\left[\frac{1}{w^4}\begin{pmatrix} 0 & 0 \\ 0 & 2i
\end{pmatrix}+\frac{1}{w}\begin{pmatrix} 0 & 0 \\ 0 & -1
\end{pmatrix}\right]\begin{pmatrix} z \\ u
\end{pmatrix}.\end{equation}
This implies that systems \eqref{mainsystem} for different
$\gamma$ are formally gauge equivalent.

We need now to deduce the same fact for the initial ODEs
\eqref{mainODE} with respect to formal gauge equivalences
$(\CC{2},0)\lr(\CC{2},0),$ which is a different issue. In order to
do so we use a strategy similar to the one used in the proof of Proposition 4.2 in \cite{divergence}, and
first consider the fundamental system of solutions for the normal form
\eqref{normalform}, which is given by
$$e^{-\frac{1}{3}w^{-3}\begin{pmatrix} 0 & 0 \\ 0 & 2i
\end{pmatrix}}\cdot w^{\begin{pmatrix} 0 & 0 \\ 0 & -1
\end{pmatrix}},$$ implying that the fundamental system of formal solutions
for \eqref{mainsystem} is given by
\begin{equation} \label{fundamsystem} \hat \Phi_\gamma(w)\cdot
e^{-\frac{1}{3}w^{-3}\begin{pmatrix} 0 & 0 \\ 0 & 2i
\end{pmatrix}}\cdot w^{\begin{pmatrix} 0 & 0 \\ 0 & -1
\end{pmatrix}},\end{equation} where
\[\hat \Phi_\gamma(w)=\begin{pmatrix}\hat f_\gamma(w) & \hat g_\gamma(w)\\
\hat h_\gamma(w) & \hat s_\gamma(w)\end{pmatrix} =
I+\sum\limits_{k\geq 2}\Phi_k w^k\]
is a matrix-valued formal power series. The columns of \eqref{fundamsystem} are
linearly independent (over the quotient field of $\CC{}[[x]]$).
% The latter means that the
% columns of \eqref{fundamsystem} are formally linearly independent
% and their formal substitution into \eqref{mainsystem} gives the
% identity.
From \eqref{fundamsystem} we conclude that the ODE
\eqref{mainODE} possesses a  fundamental system of formal solutions
$\left\{\hat f_\gamma(w),\hat g_\gamma(w)\cdot
w^{-1}\cdot\exp\left(-\frac{2i}{3}w^{-3}\right)\right\}$ for two
formal power series
\begin{equation}\label{formfg}
\hat f_\gamma(w)=1+O(w),\quad \hat g_\gamma(w)=w+O(w^2).\end{equation}
The expansion of
$\hat g_\gamma$ can be deduced from
$$
w^3\left(\hat
g_\gamma(w)w^{-1}\exp\left(-\frac{2i}{3}w^{-3}\right)\right)'=
w^{-1}\hat s_\gamma(w)\exp\left(-\frac{2i}{3}\right),
$$
which holds by the initial substitution $u=z'w^3$,
and since $\hat s_\gamma(w)=1+O(w)$, we get
 $ \ord_0 \hat g_\gamma=1$.
 Hence we can scale $\hat g_\gamma(w)$ to obtain
 $\hat g_\gamma(w)=w+O(w^2)$.

We set
\begin{equation}\label{themap}
\hat\chi(w):=\frac{1}{\hat f_\gamma(w)},\ \
\hat\tau(w):=w\left(1-\frac{3}{2i}w^{3}\ln\frac{\hat
g_\gamma(w)}{w\hat f_\gamma(w)}\right)^{-\frac{1}{3}}.
\end{equation}
In view of \eqref{formfg}, $\hat\tau(w)$ is a well defined formal
power series of the form $w+O(w^{5})$, and $\hat\chi(w)$ is a well
defined formal power series of the form $1+O(w)$. We claim now that
\begin{equation}\label{equivalence}
(z,w) \longrightarrow \left(\hat\chi(w)z, \hat\tau(w) \right)
\end{equation}
is the desired formal gauge transformation of class
$\mathcal{FG}_4$, sending $\mathcal E_\gamma$ into $\mathcal E_0$.

As it is shown in
\cite{arnoldgeom}, if two functions $z_1(w),z_2(w)$ are some
linearly independent holomorphic solutions of a second order
linear ODE $z''=p(w)z'+q(w)z$, then the transformation
$z\longrightarrow\frac{1}{z_1(w)}z,\,w\longrightarrow\frac{z_2(w)}{z_1(w)}$
transfers the initial ODE into the simplest ODE $z''=0$. The same
fact can be verified, by a simple computation, for certain spaces of formal series (as soon as all above operations
are well defined). In particular, the transformation
\begin{equation}\label{straighten}
z\longrightarrow\frac{1}{\hat f_\gamma(w)}z,\,
w\longrightarrow\frac{\hat g_\gamma(w)}{w\hat f_\gamma(w)}\exp\left(-\frac{2i}{3}w^{-3}\right)
\end{equation}
transforms formally $\mathcal E_\gamma$ into $z''=0$, and
\begin{equation}\label{straighten0}
z\longrightarrow z,\,
w\longrightarrow\exp\left(-\frac{2i}{3}w^{-3}\right)
\end{equation}
transforms $\mathcal E_0$ into $z''=0$. It follows then that the
formal substitution of \eqref{equivalence} into
\eqref{straighten0} gives \eqref{straighten}. Since the chain rule
agrees with the above formal substitutions, this shows that
\eqref{equivalence} transfers $\mathcal E_\gamma$ into $\mathcal
E_0$.
This proves the proposition.
\end{proof}

On the other hand, the ODEs $\mathcal E_0$ and $\mathcal E_\gamma$
with $\gamma\neq 0$ are different from the analytic point of view,
even though \it all \rm $\mathcal E_\gamma,\,\gamma\in\RR{}$ have
trivial monodromy.

\begin{proposition}\label{pro:trivialmon}

i) For any $\gamma\in\RR{}$ the ODE $\mathcal E_\gamma$
has a trivial monodromy; ii) For any $\gamma\in\RR{}\setminus\{0\}$ the ODE
$\mathcal E_\gamma$ has no non-zero holomorphic solutions, while
$\mathcal E_0$ has the holomorphic solution $z\equiv 1$.
\end{proposition}

\begin{proof} We first obtain the monodromy matrix for an arbitrary system
\eqref{mainsystem}. In order to do that we consider $\infty$ as an
isolated singular point for \eqref{mainsystem} and perform the
change of variables $t:=\frac{1}{w}$. We obtain the system
\begin{equation} \label{systematinfinity}
\begin{pmatrix} z \\ u
\end{pmatrix}'=\left[t^{2}\begin{pmatrix} 0 & 0 \\ 0 & -2i
\end{pmatrix}+t\begin{pmatrix} 0 & -1 \\ 0 & 0
\end{pmatrix}+\frac{1}{t}\begin{pmatrix} 0 & 0 \\ -\gamma & 1
\end{pmatrix}\right]\begin{pmatrix} z \\ u
\end{pmatrix}\end{equation} with an isolated \it Fuchsian \rm singularity
at $t=0$. The singular points of  \eqref{mainsystem} in
$\overline{\CC{}}$ are   $w=0$ and $w=\infty$, hence it is
sufficient to prove that the monodromy matrix at $t=0$ for each
system \eqref{systematinfinity} is the identity. To compute the
monodromy we apply the Poincare-Dulac normalization procedure for
Fuchsian systems (see e.g. \cite{ilyashenko}, \autoref{thm:ehwrong}6.15).
Note that the residue matrix
\[R_\gamma=\begin{pmatrix} 0 & 0 \\
-\gamma & 1
\end{pmatrix}\]
for
\eqref{systematinfinity} at $t=0$ has eigenvalues $0$ and $1$ and
hence is resonant. However, the only possible resonances of this
system correspond to the resonant monomials with \it zero \rm
degree in $t$, which are already removed from
\eqref{systematinfinity}. All higher degree monomials can be
removed from \eqref{systematinfinity} by the Poincare-Dulac
procedure after diagonalizing the residue matrix $R_\gamma$. Hence
the normal form of the system \eqref{systematinfinity} at $t=0$ is
the diagonal Euler system
\[\begin{pmatrix} z \\ u
\end{pmatrix}'=\frac{1}{t}\begin{pmatrix} 0 & 0 \\ 0 & 1
\end{pmatrix}\begin{pmatrix} z \\ u
\end{pmatrix},\]
which has trivial monodromy (as its solutions
are given by $z=c_1,\,u=c_2t$ for arbitrary constants
$c_1,c_2\in\CC{}$).  Convergence of the Poincare-Dulac normalizing
transformation in the Fuchsian case now implies i).

To prove ii) we substitute the formal power series
$h(w)=\sum\limits_{j\geq 0}a_jw^j$ into the ODE \eqref{mainODE}
with $\gamma\neq 0$ and obtain
\begin{equation}
\label{reccurent}
\begin{aligned}
a_1 &= -\frac{\gamma
a_0}{2i},\\
a_2 &= - \frac{\gamma a_1}{4i},  \\
%$k(k-1)a_k=2i(k+3)a_{k+3}-4ka_k+\gamma a_{k+2}$ for $k\geq 0$
%which we rewrite as
 a_{k+3}&=\frac{1}{2i}ka_k-\frac{\gamma}{2i(k+3)}a_{k+2}, &k\geq 0.
\end{aligned}
\end{equation}
 Clearly, $a_0=0$
implies $h\equiv 0$ so that we assume $a_0=1$ and get
$a_1=-\frac{\gamma}{2i},\,a_2=-\frac{\gamma^2}{8}.$ Note that
there exists no $s> 0$ such that $a_s\neq 0$ and $a_k=0$ for all
$k>s$, as follows from the relation \eqref{reccurent} with $k=s$.
Let \[p:=\sup\limits_{a_k\neq
0}\left|\frac{a_{k+2}}{a_k}\right|,\quad
q:=\sup\limits_{a_k\neq
0}\left|\frac{a_{k+3}}{a_k}\right|. \] If either  $p=+\infty$ or $q=+\infty$,
then the power series $h(w)$ is divergent for all $w\neq 0$;
 however, if we had $p,q<+\infty$,  \eqref{reccurent} would imply that
$k\leq \frac{p|\gamma|}{k+3}+2q$ for large $k$, which is
impossible. This proves the proposition.
\end{proof}

In fact, \autoref{pro:trivialmon} proves that all the ODEs $\mathcal
E_\gamma$ are pairwise holomorphically gauge equivalent near the
singular point $w=\infty$. We can also formulate

\begin{corollary}\label{cor:trivialmon} The nonminimal real hypersurfaces $M_\gamma$
associated with the ODEs $\mathcal E_\gamma$ have a trivial
monodromy in the sense of \cite{nonminimal} for all
$\gamma\in\RR{}$.\end{corollary}

\begin{proof} Let $h_1(w),h_2(w)$ be two linearly independent
solutions of an ODE $\mathcal E_\gamma$, defined in
$\CC{}\setminus\{0\}$. \autoref{pro:trivialmon} implies that $h_1$ and
$h_2$ are single-valued.
We now use the fact (see  \cite{arnoldgeom}) that one of the possible
mappings of the linear ODE $\mathcal E_\gamma$ onto the
ODE $z''=0$ is given by the single-valued gauge transformation
 \[ z\mapsto\frac{1}{h_2(w)}z,\quad w\mapsto \frac{h_1(w)}{h_2(w)}. \]
Hence
this  mapping takes the associated hypersurface $M_\gamma$ at a Levi
nondegenerate point $(z_0,w_0)$ onto a quadric $\mathcal
Q\subset\CP{2}$, and we conclude that the monodromy of $M_\gamma$
is trivial. \end{proof}

\begin{remark} \autoref{cor:trivialmon}, compared with \autoref{thm:nonequivalent} below, shows that
the monodromy  does not help to decide whether
irregularity phenomena for CR-mappings between given nonminimal
hypersurfaces appear, be it divergence as
 discovered in \cite{divergence}, or smoothness without analyticity, as in the present
paper. \end{remark}

We now fix some $\gamma\in\RR{}\setminus\{0\}$ and apply Sibuya's sectorial normalization theorem
(see \autoref{sec:prelim})
to connect formal and sectorial equivalences of
$\mathcal E_\gamma$ with $\mathcal E_0$. The separating rays for
each of the systems \eqref{mainsystem} are determined by
$\re\left(\frac{2i}{w^{3}}\right)=0$, so that we get the six rays
$$\left\{w=\pm \RR{+},\,w=\pm \RR{+}e^{\frac{\pi i}{3}},\,w=\pm
\RR{+}e^{-\frac{\pi i}{3}}\right\}.$$ It follows
from Sibuya's theorem that the formal matrix function
\[\hat \Phi_\gamma(w)=\begin{pmatrix}\hat f_\gamma(w) & \hat g_\gamma(w)\\
\hat h_\gamma(w) & \hat s_\gamma(w)\end{pmatrix}\] introduced
in \eqref{fundamsystem} admits (for some $r>0$) unique sectorial asymptotic
representatations $\Phi_\gamma^\pm(w)\sim\hat \Phi_\gamma$ in
sectors $S_{\pi/3,r}^\pm$, respectively, for  functions
$\Phi_\gamma^\pm(w)$ which are holomorphic in $S_{\pi/3,r}^\pm$. Accordingly, we obtain by formulas, identical to \eqref{themap}, two functions
$\chi(w),\tau(w)$, asymptotically represented in both sectors $S_{\pi/3,r}^\pm$ by the functions
$\hat\chi(w),\hat\tau(w)$,  respectively.

In what follows we use the notation $\mathcal{SG}^\pm_m$ for the class of gauge transformations $z\to zf(w),\,w\to g(w)$ such that the functions $f$ and $g$ are holomorphic in a sector $S_{\alpha,r}^\pm$ respectively for some $r>0,0<\alpha<\frac{\pi}{2}$ and, in addition, having in the sector $S_{\alpha,r}^\pm$ asymptotic power series representations
with the properties   $f(w)=1+O(w),g(w)=w+O(w^{m+1})$. Considering then the formal gauge equivalence
between $\mathcal{E}_\gamma$ and $\mathcal{E}_0$, given by \eqref{equivalence}, we see from
the proof of \autoref{pro:formalequivalODE} that the map
\begin{equation} \label{sectorialmap} (z,w)\lr
(z\chi^\pm(w),\tau^\pm(w))
\end{equation}
is of class $\mathcal{SG}^\pm_4$ and, moreover,
transfers the ODE $\mathcal E_\gamma$ into
$\mathcal E_0$. The latter statement follows from the fact that,
according to Sibuya's theorem, the map $$\begin{pmatrix}z\\
u\end{pmatrix}\mapsto
\Phi_\gamma^\pm(w)\cdot\begin{pmatrix}z\\
u\end{pmatrix},\quad w\mapsto w$$
transforms  the system \eqref{mainsystem}
into its normal form \eqref{normalform}. Arguing then identically
to the proof of \autoref{pro:formalequivalODE}, we see that \eqref{sectorialmap}
transfers $\mathcal E_\gamma$ into $\mathcal E_0$.

We also need a uniqueness statement for normalized gauge
equivalences between $\mathcal E_\gamma$ and $\mathcal E_0$. We note that a statement
similar  to \autoref{pro:unique} below for
the  systems \eqref{mainsystem}, associated
with the ODEs $\mathcal E_\gamma$ and $\mathcal E_0$ respectively, follows
directly from the uniqueness in Sibuya's theorem. However, gauge
equivalences of ODEs is a \it different \rm issue, which needs a separate treatment.

\begin{proposition}\label{pro:unique} The only transformation $F\in\mathcal {SG}^\pm_4$, transferring $\mathcal
E_\gamma$ into $\mathcal E_0$, is given by \eqref{sectorialmap}.
\end{proposition}

\begin{proof} It is
sufficient to prove that the unique transformation $F\in\mathcal
{SG}_4$, transferring $\mathcal E_0$ into itself, is the identity.
If a  gauge transformation $F=(zf(w),g(w))\in\mathcal
{SG}^\pm_4$ preserves $\mathcal E_0$, we know that
$\{z=\frac{1}{f(w)}\}$ is (locally) the graph of a solution of
$\mathcal E_0$ (as it is the preimage of the graph $\{z=1\}$).
Since each solution of $\mathcal E_0$ is a linear combination
of $1$ and
$\exp{\left(-\frac{2i}{3w^3}\right)}$, and $\frac{1}{f}$ has
an asymptotic expansion of the form $1+O(w)$ in a
sector  $S^\pm_{\alpha,r}$, we conclude that $f\equiv 1$. Thus
$F=(z,g(w))$. Substituting $F$ into $\mathcal E_0$, we get in the
preimage
$$\frac{1}{(g')^2}z''-\frac{g''}{(g')^3}z'=\frac{1}{g'}\left(\frac{2i}{g^4}-\frac{4}{g}\right)z'.$$ Since $\mathcal E_0$ is
preserved, we obtain
\begin{equation}
\label{newODE2}g'\left(\frac{2i}{g^4}-\frac{4}{g}\right)+\frac{g''}{g'}=\frac{2i}{w^4}-\frac{4}{w}
.\end{equation} We now argue similarly to the proof of Proposition
4.4 in \cite{divergence} and study the ODE \eqref{newODE2}.
Assuming that $g(w)\not\equiv w$,~\eqref{newODE2} can be rewritten
as a  differential relation
$$2i\left(-\frac{1}{3g^{3}}\right)'+2i\left(\frac{1}{3w^{3}}\right)'+(\ln
g')'-4\left(\ln\frac{g}{w}\right)'=0,$$ which gives
$-\frac{2i}{3}\left(\frac{1}{g^{3}}-\frac{1}{w^{3}}\right)+\ln
g'-4\ln\frac{g}{w}=C_1$ for some constant $C_1\in\CC{}$. It
follows then that the function
$-\frac{1}{3}\left(\frac{1}{g^{3}}-\frac{1}{w^{3}}\right)$ has an
asymptotic representation by a formal power series in a sector
$S^\pm_{\alpha,r}$, and a straightforward computation shows that
the substitution
$-\frac{1}{3}\left(\frac{1}{g^{3}}-\frac{1}{w^{3}}\right):=u$
transforms the latter equation for $g$ into
$2iu+\ln(w^4u'+1)=C_1$. Shifting $u$, we get the equation
$2iu+\ln(w^4u'+1)=0$, where $u(w)$ is represented in
$S^\pm_{\alpha,r}$ by a formal power series with zero free term.
Hence we finally obtain the following meromorphic first order ODE
for the shifted function $u(w)$:
\begin{equation}\label{ODEforu}u'=\frac{1}{w^4}(e^{-2iu}-1).\end{equation} However,
if $u\not\equiv 0,$ \eqref{ODEforu} can be represented as
$-\frac{1}{2i}u'\left(\frac{1}{u}+H(u)\right)=\frac{1}{w^4}$,
where $H(t)$ is a holomorphic at the origin function. Hence we get
that the logarithmic derivative $\frac{u'}{u}$ is asymptotically
represented in $S^\pm_{\alpha,r}$ by a formal Laurent series
$-\frac{2i}{w^4}+...$, where the dotes denote a formal power
series in $w$. But this clearly contradicts the existence of an
asymptotic representation of $u(w)$ by a power series in
$S^\pm_{\alpha,r}$. Hence $u\equiv 0$, and, returning to the
unknown function $g$, we get $\frac{1}{g^{3}}-\frac{1}{w^{3}}=C$
for some constant $C\in\CC{}$, so that
$g(w)=\frac{w}{(1+Cw^{3})^\frac{1}{3}}$. Taking into account the
asymptotic representation $g(w)=w+O(w^{5})$, we conclude that
$C=0$ and $g(w)=w$. This proves the proposition.
\end{proof}

Let $\mathcal S=\{w=\rho(z,\bar\xi,\bar\eta)\}$ be a (general)
Segre family in a polydisc
$\Delta_\delta\times\Delta_\varepsilon$.
%I think here also it is important to keep notation ok
According to \cite{divergence}, we call the complex submanifold
\begin{equation}\label{foliated}
\mathcal M_{\rho}=\left\{(z,w,\xi,\eta)\in\Delta_\delta\times\Delta_\varepsilon\times\Delta_\delta\times\Delta_\varepsilon:
\,w=\rho(z,\xi,\eta)\right\}\subset\CC{4},
\end{equation}
   \emph{the foliated submanifold} associated with  $\rho$.
 If $\mathcal S$ is associated with a $\mathcal P_0$-ODE with an isolated
 meromorphic singularity
 $\mathcal E$, then $\mathcal M_{\rho}$ is called \emph{the associated foliated submanifold of} $\mathcal E$,
and if $\mathcal S$ is $m$-admissible, then
 $\mathcal M_{\mathcal S}$  is said to be $m$-admissible (as before, we use
 the unique $\rho$ satisfying the conditions of $m$-admissiblity).
 If $\mathcal{S}$ is the Segre family of a
real hypersurface $M\subset\CC{2}$, then the associated foliated
submanifold is simply the complexification of $M$.

A foliated
submanifold $\mathcal M_{\mathcal S}$ possesses two natural
foliations, induced by the projections on the
first two and the last two coordinates, respectively; i.e.
the first one is the initial foliation~$\mathcal S$
with leaves $\{(z,w,\xi,\eta)\in\mathcal M_{\mathcal
S}:\,\xi=\xi_0,\,\eta=\eta_0 \}$. The second one is the family of
dual Segre varieties with leaves $\{(z,w,\xi,\eta)\in\mathcal
M_{\mathcal S}:\,z=z_0, \, w=w_0 \}$. It is natural to consider the
so-called \it coupled equivalences \rm between foliated
submanifolds. The latter have the form
\begin{equation}\label{twin}
(z,w,\xi,\eta)\longrightarrow(F(z,w),G(\xi,\eta)) ,
\end{equation}
where $F(z,w),G(\xi,\eta)$ are  biholomorphisms
$(\CC{2},0)\longrightarrow(\CC{2},0)$, and thus preserve both the
foliated submanifolds and the above two foliations.

Our next goal is to show that for any $F\in\mathcal{SG}^\pm_m$,
conjugating two linear $\mathcal P_0$-ODEa with an isolated
 meromorphic singularity, a transformation
$G\in\mathcal{SG}^\pm_m$ can be chosen in such a way that the
direct product $(F,G)$ conjugates the associated foliated
submanifolds. We first note that for any $m$-admissible foliated
submanifold $\mathcal M$ each of the intersections $\mathcal
M\cap\{\pm\re \eta>0\}$ lies in a domain
$G^\pm_{R,r\alpha}=\Delta_R\times
S^{\pm}_{\alpha,r}\times\Delta_R\times S^{\pm}_{\alpha,r}$ for
sufficiently small $R,r,\alpha$. Also note that for any
$(F,G)\in\mathcal{SG}^\pm_m\times\mathcal{SG}^\pm_m$ and
sufficiently small $R,r,\alpha$ the image of a domain
$G^\pm_{R,r,\alpha}$ satisfies
$G^\pm_{R_1,r_1,\alpha_1}\subset(F,G)(G^\pm_{R,r,\alpha})\subset
G^\pm_{R_2,r_2,\alpha_2}$. Moreover, the asymptotic expansion of
$F,G$ implies that, by choosing $R,r,\alpha$ small enough, one can
make $(F,G)$ arbitrarily close to the identity in the sense that
the mapping $(F,G)-\mbox{Id}$ is Lipschitz with an arbitrarily
small constant. Hence we can assume that $(F,G)$ is biholomorphic in
$G^\pm_{R,r,\alpha}$. For the inverse mapping $(F^{-1},G^{-1})$ one
has $(F^{-1},G^{-1})\in
\mathcal{SG}^\pm_m\times\mathcal{SG}^\pm_m$. Thus the image
$(F,G)(\mathcal M\cap\{\pm\re \eta>0\})$ is a holomorphic graph of
kind $w=\varphi(z\xi,\eta)$ for some $\varphi\in\mathcal
O(\Delta_{R^2}\times S^{\pm}_{\alpha,r})$.
Indeed, $(F,G)(\mathcal M\cap\{\pm\re \eta>0\})$ is obtained by substituting
$(F^{-1},G^{-1})$ into the defining equation of $\mathcal M$ and
applying the implicit function theorem, so that (in a sufficiently
small domain $G^\pm_{R_1,r_1,\alpha_1}$) $(F,G)(\mathcal
M\cap\{\pm\re \eta>0\})$ can be represented as a graph
$w=\varphi(z\xi,\eta)$ locally near each point of it. This implies
the required global representation. Moreover, $\varphi(z\xi,\eta)$
has the form $\varphi(z\xi,\eta)=\eta
e^{i\eta^{m-1}\varphi^*(z\xi,\eta)}$ with
$\varphi^*(z\xi,\eta)=\sum\limits_{k\geq
0}\varphi^*_{k}(\eta)z^k\xi^k$ for some $\varphi^*_k\in\mathcal
O(S^{\pm}_{\alpha_1,r_1})$. The series converges uniformly on
compact subsets in $S^{\pm}_{\alpha_1,r_1}$ and each $\varphi^*_j$
has some asymptotic expansion in a, possibly, smaller sector
$S^{\pm}_{\alpha_1,r^*}$ (the latter fact follows from the
implicit function theorem for asymptotic series, applied for fixed
$z,\xi$).

We now need

\begin{proposition}\label{pro:uniqueG} For any $m$-admissible foliated submanifold $\mathcal
M$ and any map $F\in\mathcal{SG}^\pm_m$ there exists a unique
$G\in\mathcal{SG}^\pm_m$ such that for the function
$\varphi^*(z,\xi,\eta)$ as above one has
$\varphi^*_{0,k}(\eta)=\varphi^*_{k,0}=0$, $\varphi^*_{1,1}(\eta)=1$, and $\varphi^*_{1,k}=\varphi^*_{k,1}=0$
for $k>1$.\end{proposition}

\begin{proof} Denote the components of the inverse sectorial mapping
$(F^{-1},G^{-1})$ by $(zf(w),g(w),\xi\lambda(\eta),\mu(\eta))$.
Then (after choosing sufficiently small domains
$G^\pm_{R,r,\alpha}$ as above) $(F,G)(\mathcal M\cap\{\pm\re
\eta>0\})$ is described as
\begin{equation}\label{twingauge}
g(w)=\mu(\eta)e^{i\mu(\eta)^{m-1}\psi(zf(w),\xi\lambda(\eta),\mu(\eta))},\end{equation}
where $w=\eta e^{i\eta^{m-1}\psi}$ is the defining equation of
$\mathcal M$. The fact that $\varphi^*_0,\varphi^*_1$, determined
by \eqref{twingauge}, have the desired form, reads (for each fixed
$\eta$) as $$g(\eta)+i\eta^m
g'(\eta)z\xi=\mu(\eta)+i\mu(\eta)^mz\xi
f(\eta)\lambda(\eta)+O(z^2\xi^2).$$ The latter is equivalent to
\begin{equation}\label{twingauge2}g(\eta)=\mu(\eta),\,\eta^m
g'(\eta)=\mu(\eta)^m f(\eta)\lambda(\eta).\end{equation} Equations
\eqref{twingauge2} determine
$\lambda(\eta),\mu(\eta)$ with the desired properties uniquely,
and this proves the proposition.
\end{proof}

Recall now that, by assumption, the mapping $F\in\mathcal{SG}^\pm_m$ transfers a
linear $\mathcal P_0$-ODE $\mathcal E$ with an isolated
 meromorphic singularity  onto another linear $\mathcal P_0$-ODE
  $\mathcal {\tilde E}$ with an isolated
 meromorphic singularity. This means that $F$ transfers germs of
leaves of the foliation $\mathcal M\cap\{\xi=const,\eta=const\}$
with $\pm\re\eta>0$ into germs of holomorphic graphs, satisfying
the ODE $\mathcal {\tilde E}$. Thus the substitution of $w=\eta
e^{\eta^{m-1}\varphi^*(z,\xi,\eta)}$ into $\mathcal {\tilde E}$,
where $\varphi^*$ is the unique defining function, obtained in
 \autoref{pro:uniqueG},  gives an identity. Fixing $\xi$ and $\eta$ and
performing the substitution, we obtain a second order ODE for the
function $\varphi(\cdot,\xi,\eta)$, identical to \eqref{findphi}.
The uniqueness of a solution for this ODE with the initial data
$\varphi(0,\xi,\eta)=0,\dot\varphi(0,\xi,\eta)=\xi$ implies that the
defining functions of $\mathcal M_{\mathcal {\tilde E}}$ and
$(F,G)(\mathcal M\cap\{\pm\re \eta>0\})$ coincide for all $\eta\in
S^\pm_{r,\alpha}$. This proves that the sectorial mapping $(F,G)$,
obtained by \autoref{pro:uniqueG}, transfers $\mathcal M\cap\{\pm\re
\eta>0\}=\mathcal M_{\mathcal E}\cap\{\pm\re \eta>0\}$ into
$\mathcal M_{\mathcal {\tilde E}}\cap\{\pm\re \eta>0\}$ (after
intersecting $M_{\mathcal E},M_{\mathcal {\tilde E}}$ with
sufficiently small polydiscs). \autoref{pro:uniqueG} also implies that
$(F,G)$ with such property is unique.

On the other hand, if a mapping $(F,G)\in \mathcal
{SG}^\pm_m\times \mathcal {SG}^\pm_m\times$ transfers $\mathcal
M_{\mathcal E}$ into $\mathcal M_{\mathcal {\tilde E}}$, where
$\mathcal E,\mathcal {\tilde E}$ are two linear $\mathcal
P_0$-ODEs with an isolated
 meromorphic singularity, then
$F$ transfers germs of leaves of the foliation $\mathcal
M_{\mathcal E}\cap\{\xi=const,\eta=const\}$ with $\pm\re\eta>0$
into germs of leaves of the foliation $\mathcal M_{\mathcal
{\tilde E}}\cap\{\xi=const,\eta=const\}$ with $\pm\re\eta>0$. This
implies that $F\in \mathcal {SG}^\pm_m$ is an equivalence of the
ODEs $\mathcal E$ and $\mathcal {\tilde E}$.

All above arguments prove the following

\begin{proposition}
Let $\mathcal E,\mathcal {\tilde E}$ be two linear $\mathcal
P_0$-ODEs with an isolated
 meromorphic singularity,
and $\mathcal M_{\mathcal{E}} $, $ M_{\tilde{\mathcal{E}}}\subset\CC{4}$ the associated
foliated submanifolds. Then there is a one-to-one correspondence
$F(z,w)\longrightarrow(F(z,w),G(\xi,\eta))$ between sectorial
equivalences $F(z,w)\in\mathcal {SG}^\pm_m$, transferring
$\mathcal E$ into $\mathcal {\tilde E}$, and  coupled sectorial
transformations $(F(z,w),G(\xi,\eta))\in\mathcal {SG}^\pm_m \times
\mathcal {SG}^\pm_m$, sending $\mathcal M_{\mathcal{E}} \cap\{\pm\re \eta>0\}$
into $\mathcal M_{\tilde{\mathcal{E}}} \cap\{\pm\re \eta>0\}$.
\end{proposition}

We are now in the position to prove our principal result.

\begin{theorem}\label{thm:nonequivalent}
For any $\gamma\neq 0$ the germs at $0$ of the real-analytic
hypersurfaces $M_\gamma$ and $M_0$, associated with the ODEs
$\mathcal E_\gamma$ and $\mathcal E_0$ respectively, are
$C^\infty$ CR-equivalent, but are holomorphically
inequivalent.\end{theorem}

\begin{proof}
In what follows we assume that the real hypersurfaces $M_\gamma$
and $M_0$, as well as their complexifications $\mathcal
M_{\mathcal E_\gamma}$ and $\mathcal M_{\mathcal E_0}$, are
intersected with a sufficiently small neighborhood of the origin,
if necessary. As was discussed above, the sectorial map $F^+\in
\mathcal {SG}^+_4$, as in \eqref{sectorialmap}, transfers
$\mathcal E_\gamma$ into $\mathcal E_0$. According to Proposition
4.7, there exists a unique $G^+\in \mathcal {SG}^+_4$, such that
$(F^+,G^+)$ transfers $\mathcal M_{\mathcal
E_\gamma}\cap\{\re\eta>0\}$ into $\mathcal M_{\mathcal
E_0}\cap\{\re\eta>0\}$. Considering now the reality condition
\eqref{reality} for  the real hypersurfaces $M_\gamma$ and
complexifying it, we conclude that every $\mathcal M_{\mathcal
E_\gamma}=(M_\gamma)^{\CC{}}$ is invariant under the
anti-holomorphic linear mapping
$\sigma:\,\CC{4}\longrightarrow\CC{4}$ given by
\begin{equation}\label{sigma}(z,w,\xi,\eta)\longrightarrow(\bar\xi,\bar\eta,\bar z,\bar
w).\end{equation}
Thus the sectorial mapping
$\sigma\circ(F^+(z,w), G^+(\xi,\eta))\circ\sigma=(\overline{
G^+}(z,w),\overline{ F^+}(\xi,\eta))\in\mathcal
{SG}^+_4\times\mathcal {SG}_4^+$ also transfers $\mathcal
M_{\mathcal E_\gamma}\cap\{\re\eta>0\}$ into $\mathcal M_{\mathcal
E_0}\cap\{\re\eta>0\}$ (here $\overline
{F^+}(z,w):=\overline{F^+(\bar z,\bar w)}$ and similarly for
$\overline {G^+}$). Now the uniqueness, given by \autoref{pro:unique},
implies $F^+(z,w)=\overline {G^+}(z,w)$. In particular, this means
that $F^+(z,w)$ transfers $M_\gamma^+=M_\gamma\cap\{\re w>0\}$
into $M_0^+=M_0\cap\{\re w>0\}$. Similarly, we get that the
sectorial mapping $F^-(z,w)$, as in \eqref{sectorialmap},
transfers $M_\gamma^-=M_\gamma\cap\{\re w<0\}$ into
$M_0^-=M_0\cap\{\re w<0\}$.

We now define the desired $C^\infty$ CR-equivalence as follows:
\begin{equation}\label{Cinfinity} F(z,w)=\left\{\begin{array}{lcl}
F^-(z,w),\,(z,w)\in M^-_\gamma\\
(z,0),\,(z,w)\in X\\
F^+(z,w),\,(z,w)\in M^+_\gamma\\
\end{array}
\right.\end{equation} The  arguments above imply that
$F(M_\gamma)\subset M_0$,  and the asymptotic expansion for the
mappings \eqref{sectorialmap} shows that $F$ is a local
$C^\infty$ diffeomorphism on $M$. Now $F$ is CR because it is actually analytic
on each of the CR-manifolds $M^+$ and and
$M^-$, and thus it satisfies the tangential CR-equations
on all of $M$.  This proves the CR-equivalence of the
germs $(M_\gamma,0)$ and $(M_0,0)$.

To prove the holomorphic nonequivalence  of $M_\gamma$ and
$ M_0$ we finally note that each local holomorphic equivalence
$\varphi:\,(M_\gamma,0)\lr(M_0,0)$ extends to a holomorphic equivalence of the
associated ODEs $\mathcal E_\gamma$ and $\mathcal E_0$ near the
singular point $(z,w)=(0,0)$, as follows from the invariance
property of Segre varieties. However , $\mathcal E_0$ has a
non-zero holomorphic solution, while $\mathcal E_\gamma$ does not
have one, which shows that such a local holomorphic equivalence does not exist.
This completely proves  the theorem.
\end{proof}

\autoref{thm:nonequivalent} enables us to give the negative answer to the
Conjecture of Ebenfelt and Huang (see Introduction).

\begin{corollary}\label{cor:noanalytic} The mapping \eqref{Cinfinity} is a $C^\infty$
CR-equivalence between the germs of the
real-analytic holomorphically nondegenerate hypersurfaces
$M_\gamma,M_0\subset\CC{2}$, which does not have the analyticity
property.\end{corollary}
%Change this wording

It is not difficult now to deduce the proof of \autoref{thm:ehwrong}. We first note that the real hypersurface
$M_0$, associated with the ODE $\mathcal E_0$, coincides with the real hypersurface $M^4_0\subset\CC{2}$,
considered in \cite{divergence} (while all $M_\gamma$ with $\gamma\neq 0$ are different from the hypersurfaces,
considered in \cite{divergence}). A detailed computation, provided in \cite{divergence} (see Section 5 there), shows
that the single-valued elementary mapping
$\Lambda'\colon(z,w)\mapsto \left(\sqrt 2 z,e^{-\frac{2i}{3}w^{-3}}\right)$
takes $M_0\setminus X$ into the
compact sphere
$S^3=\{|Z|^2+|W|^2=1\}\subset\CC{2}$ (where
and $X=\{w=0\}$).
It follows from \eqref{Cinfinity} and
\eqref{sectorialmap}
that, for any fixed $\gamma$, a mapping of
$M_\gamma\setminus X$ into $S^3$ has (locally) the form
$\Lambda \colon (z,w)\mapsto (z\mu(w),\nu(w)).$ According to the
globalization result \cite{nonminimal}, $\mu(w),\nu(w)$ extend to
analytic mappings $\CC{}\setminus\{0\}\to\CP{1}$. Since the ODE
$\mathcal E_\gamma$ has a trivial monodromy, so does the mapping
$\Lambda$ and we conclude that the extensions of both $\mu(w)$
and $\nu(w)$ are single-valued.

\begin{proof}[Proof of \autoref{thm:ehwrong}.]
For $n=k=1$ the result is just the one proved in \autoref{thm:nonequivalent}.
For $k=1$ and $n>1$ (which corresponds to hypersurfaces in $\CC{n+1}$) we
consider the above hypersurfaces $M_\gamma,M_0\subset \CC{2}$ and
write them near the origin as $v=\Theta_\gamma(z\bar z,u)$ and
$v=\Theta_0(z\bar z,u)$ respectively (here $w=u+iv$). The mapping
$F(z,w)=(z\chi(w),\tau(w))$, as in \eqref{Cinfinity}, provides a
$C^\infty$ CR-equivalence between $(M_\gamma,0)$ and $(M_0,0)$.
We now define
$$M:=\left\{v=\Theta_\gamma(z_1\bar z_1+...+z_n\bar z_n,u)\right\}\subset\CC{n+1},\quad M':=\left\{v=\Theta_0(z_1\bar
z_1+...+z_n\bar z_n,u)\right\}\subset\CC{n+1}.$$ Then it is
immediate that the mapping $$H_n \colon (z_1,...,z_n,w)\to
(\chi(w)z_1,...,\chi(w)z_n,\tau(w))$$ is a $C^\infty$
CR-equivalence between $M$ and $M'$.

We claim that $(M,0)$ and $(M',0)$ are biholomorphically
inequivalent.  As in the case of
$\CC{2}$, the mapping
 $\Lambda'_n \colon (z_1,...,z_n,w)\to\left(\sqrt 2 z_1,...,\sqrt 2 z_n,e^{-\frac{2i}{3}w^{-3}}\right)$
maps $M'\setminus\{w=0\}$ into the sphere
 $ S^{2n-1}=\left\{|Z_1|^2+...+|Z_n|^2+|W|^2=1\right\}$, and the
mapping $\Lambda_n \colon (z_1,...,z_n,w)\to \left(\mu(w)z_1,...,\mu(w)z_n,\nu(w)\right)$ maps
$M\setminus\{w=0\}$ into $S^{2n-1}$. The pullback of the Segre
family of the sphere $S^{2n-1}$ by the mapping $\Lambda'_n$ provides
an extension of the Segre family of $M'$ consisting of
horizontal hyperplanes $\{w=const\}$ and the $(n+1)$-parameter
family of complex hypersurfaces
$$\left\{a_1z_1+...+a_nz_n+be^{-\frac{2i}{3}w^{-3}}+c=0,\,|a_1|^2+..+|a_n|^2\neq 0\right\}.$$
Similarly, for $M$ we get the $(n+1)$-parameter family of complex hypersurfaces
\begin{equation}\label{badfamily}\left\{\mu(w)(a_1z_1+...+a_nz_n)+b\nu(w)+c=0,\,|a_1|^2+..+|a_n|^2\neq
0\right\},\end{equation} where $\mu(w),\nu(w)$ are as above. In
particular, for the real hypersurface $M'$ and $b=0$ we get, for
appropriate values of $(a_1,...,a_n)$, an $n$-parameter family of
complex hypersurfaces (in fact, complex hyperplanes), defined in
an open neighborhood of the origin and intersecting the complex
locus $X=\{w=0\}$ of $M'$ transversally. In case $M$ and $M'$ are
locally biholomorphically equivalent at the origin, a similar
$n$-parameter family must exist for $M$ as well. However, from the
form of \eqref{badfamily} this is possible only if one of the
functions $\frac{1}{\mu(w)}$ or $\frac{\nu(w)}{\mu(w)}$ extends to
$w=0$ holomorphically (as a mapping into $\CC{}$). In both cases
we conclude that the extended Segre family of the above real
hypersurface $M_\gamma\subset\CC{2}$
contains a graph $z=f(w)$, $f(w)\not\equiv 0$,  with $f$ holomorphic
near the origin, and hence the ODE
$\mathcal E_\gamma$ has a non-zero holomorphic solution, which is
a contradiction. This completes the proof for $k=1, n>1$.

Finally, for the case $k>1$ and CR-dimension $n\geq 1$ we argue
similarly to the proof of the analogous statement in
\cite{divergence} in the case $k>1$ and consider the
holomorphically nondegenerate CR-submanifolds $P=M\times
\Pi_{k-1}$ and $P'=M'\times \Pi_{k-1}$, where
$M,M'\subset\CC{n+1}$ are chosen from the hypersurface case and
$\Pi_{k-1}\subset\CC{k-1}$ is the totally real plane $\im
W=0,W\in\CC{k-1}$. Then the direct product of the above mapping
$H_n$ and the identity map gives a $C^\infty$ CR-equivalence
between $P$ and $P'$. To show  that $P$ and $P'$ are inequivalent
holomorphically, we  denote the coordinates in $\CC{n+k}$ by
$(Z,W),\,Z\in\CC{n+1},W\in\CC{k-1}$ and note that, since $\Pi$ is
totally real, for each holomorphic equivalence
$$
(\Phi(Z,W),\Psi(Z,W)):\,(M\times \Pi_{k-1},0)\lr (M'\times
\Pi_{k-1},0) ,
$$
one has $\Psi(Z,W)=\Psi(W)$ for a vector power series $\Psi(W)$
with real coefficients such that $\Psi(0)=0$. Since the initial mapping
$(\Phi(Z,W),\Psi(Z,W))$ is invertible at $0$, we conclude that the
mapping $\Phi(Z,0):\,(\CC{n},0)\lr(\CC{n},0)$ is invertible at $0$
as well, and since $(\Phi(Z,W),\Psi(W)):\,(M\times \Pi_{k-1},0)\lr
(M'\times \Pi_{k-1},0)$, the map $\Phi(Z,0)$ is a local
equivalence between $(M,0)$ and $(M',0)$. Now the desired
statement is obtained from the hypersurface case. \end{proof}

\section{Applications to CR-automorphisms}

From the results of the previous section we are able to obtain
various somehow paradoxical phenomena for CR-automorphisms of
nonminimal real-analytic hypersurfaces. We start by recalling the
form of the infinitesimal automorphism algebra of a point in the
sphere $S^3$. This $8$-dimensional real Lie algebra is spanned by
the vector fields \begin{equation}\label{spherealgebra}
\begin{aligned}
X_1&=iZ\frac{\partial}{\partial Z},\\
X_2&=iW\frac{\partial}{\partial W},\\
X_3&=W\frac{\partial}{\partial Z}-Z\frac{\partial}{\partial W},\\
X_4&=iW\frac{\partial}{\partial Z}+iZ\frac{\partial}{\partial W},\\
 X_5&=(1-Z^2)\frac{\partial}{\partial Z}-ZW\frac{\partial}{\partial W},\\
 X_6&=i(1+Z^2)\frac{\partial}{\partial Z}+iZW\frac{\partial}{\partial W},\\
 X_7&=-ZW\frac{\partial}{\partial Z}+(1-W^2)\frac{\partial}{\partial W},\\
 X_8&=iZW\frac{\partial}{\partial Z}+i(1+W^2)\frac{\partial}{\partial W}.
 \end{aligned} \end{equation} Consider now a
hypersurfaces $M_\gamma$ with $\gamma\neq 0$ as well as the
hypersurface $M_0$ (see Section 4). Denote by
$F(z,w)=(z\chi(w),\tau(w))$ the non-analytic $C^\infty$
CR-equivalence \eqref{Cinfinity} between $(M_\gamma,0)$ and
$(M_0,0)$, and by $\Lambda$ and $\Lambda'$ the above described
single-valued mappings of $M_\gamma\setminus X$ and $M_0\setminus
X$ respectively into the sphere $S^3$. Substituting the elementary
mapping $\Lambda'$ into \eqref{spherealgebra} for $w\neq 0$, it is
straightforward to check that the vector fields
$\Lambda_*'X_1,\Lambda_*'X_2,\Lambda_*'X_5,\Lambda_*'X_6$ extend
to elements of $\mathfrak{hol}^\omega\,(M_0,0)$ and, moreover,
\begin{gather}\label{M0algebra} \Lambda_*'X_1=iz\frac{\partial}{\partial z},\,\Lambda_*'X_2=2w^4\frac{\partial}{\partial
w},\\
\notag\Lambda_*'X_5=\frac{1}{\sqrt
2}(1-2z^2)\frac{\partial}{\partial z}+2\sqrt
2izw^4\frac{\partial}{\partial w},\,\Lambda_*'X_6=\frac{i}{\sqrt
2}(1+2z^2)\frac{\partial}{\partial z}+2\sqrt
2zw^4\frac{\partial}{\partial w}\end{gather} It is also
straightforward that no non-zero linear combination of
$\Lambda_*'X_3,\Lambda_*'X_4,\Lambda_*'X_7,\Lambda_*'X_8$ extends
 neither to an element of $\mathfrak{hol}^\omega\,(M_0,0)$ nor to an
element of $\mathfrak{hol}^\infty\,(M_0,0)$, so that
$$\mathfrak{hol}^\omega\,(M_0,0)=\mathfrak{hol}^\infty\,(M_0,0)=\mbox{span}_{\RR{}}<\Lambda_*'X_1,\Lambda_*'X_2,\Lambda_*'X_5,\Lambda_*'X_6>.$$

Since the mapping $F(z,w)$ provides a $C^\infty$ CR-equivalence
 between $(M_\gamma,0)$ and $(M_0,0)$, we have
 $\mathfrak{hol}^\infty\,(M_\gamma,0)=F_*(\mathfrak{hol}^\infty\,(M_0,0))$.
 Substitution of $F$ into
 $\mbox{span}_{\RR{}}<\Lambda_*'X_1,\Lambda_*'X_2,\Lambda_*'X_5,\Lambda_*'X_6>$
 gives $\mbox{span}_{\RR{}}<Y_1,Y_2,Y_5,Y_6>$, where
 \begin{equation}\label{M0algebra2}  \begin{aligned}
 Y_1&=iz\frac{\partial}{\partial z},\\
 Y_2&=-\frac{2\tau^4\chi'}{\chi\tau'}z
 \frac{\partial}{\partial z}
 +\frac{2\tau^4}{\tau'}
 \frac{\partial}{\partial w},\\
 Y_5&=\left(\frac{1}{\sqrt 2}\frac{1}{\chi}-2\chi z^2-2\sqrt 2i\frac{\chi'\tau^4}{\tau'}z^2\right)\frac{\partial}{\partial z}+2\sqrt 2i\frac{\chi\tau^4}{\tau'}z\frac{\partial}{\partial w},\\
  Y_6&=\left(
  \frac{i}{\sqrt 2}
  \frac{1}{\chi}+
  2i\chi z^2
  -2\sqrt 2\frac{\chi'\tau^4}{\tau'}
  z^2
  \right)
  \frac{\partial}{\partial z}
  +2\sqrt 2\frac{\chi\tau^4}{\tau'}z
  \frac{\partial}{\partial w}
  \end{aligned} \end{equation} and $\chi=\chi(w),\,\tau=\tau(w)$.

In what follows we denote by $\mathcal O_0$ the space of
germs of holomorphic functions at the origin. Recall
that the restrictions $\chi^\pm,\tau^\pm$ of the functions
$\chi(w),\tau(w)$ to the sectors $S^\pm_{\frac{\pi}{3}}$
respectively have the asymptotic representations
$\hat\chi(w)=1+O(w),\,\hat\tau(w)=w+O(w^5)$ (see the proof of
\autoref{pro:formalequivalODE}). Note that at least one of the functions
$\chi(w),\tau(w)$ does not belong to $\mathcal O_0$, because
 the CR-equivalence $F(z,w)$ is not holomorphic at the
origin. Also note that the vector field $Y_1$ extends to the
origin holomorphically. We now consider three cases.

Assume first that $\tau(w)\not\in\mathcal O(0),\chi(w)\in\mathcal
O_0$. Then the function
$\frac{\tau^4}{\tau'}=-1/\left(3(\tau^{-3})'\right)\not\in\mathcal
O_0$ and, considering the $\frac{\partial}{\partial w}$
components of the three vector fields $Y_2,Y_3,Y_4$, we conclude
that no nontrivial real linear combinations of $Y_2,Y_3,Y_4$
extends to the origin holomorphically. Thus
$\mathfrak{hol}^\omega\,(M_\gamma,0)=\mbox{span}_{\RR{}}<Y_1>$.

Assume next $\tau(w)\in\mathcal O_0,\chi(w)\not\in\mathcal O_0$.
Then the functions
$\frac{1}{\chi},\frac{\chi'}{\chi}\not\in\mathcal O_0$ and,
considering the $\frac{\partial}{\partial z}$ components of the
three vector fields $Y_2,Y_3,Y_4$, we conclude that all real
non-zero linear combinations of $Y_2,Y_3,Y_4$ do not extend to the
origin holomorphically. Thus
$\mathfrak{hol}^\omega\,(M_\gamma,0)=\mbox{span}_{\RR{}}<Y_1>$.

Finally, assume $\tau(w)\not\in\mathcal
O_0,\chi(w)\not\in\mathcal O_0$. Then
$\frac{\tau^4}{\tau'},\frac{1}{\chi},\frac{\chi'}{\chi}\not\in\mathcal
O_0$ and, considering the $\frac{\partial}{\partial w}$ component
for $Y_2$ and the $\frac{\partial}{\partial z}$ component for
$Y_3,Y_4$, we also come to the conclusion
$\mathfrak{hol}^\omega\,(M_\gamma,0)=\mbox{span}_{\RR{}}<Y_1>$.

We summarize our arguments in

\begin{theorem}\label{thm:dimensionsinfinitesimal}
For  $\gamma\neq 0$ the hypersurface $M_\gamma$ defined  above satisfies

$$\mbox{dim}\,\mathfrak{hol}^\omega\,(M_\gamma,0)=1,\,\mbox{dim}\,\mathfrak{aut}^\omega\,(M_\gamma,0)=1,$$
while
$$\mbox{dim}\,\mathfrak{hol}^\infty\,(M_\gamma,0)=4,\,\mbox{dim}\,\mathfrak{aut}^\infty\,(M_\gamma,0)=2.$$
\end{theorem}

We are now in the position to prove our second main result.

\begin{proof}[Proof of \autoref{thm:infinitesimal}]
The strategy of the proof is similar to that for \autoref{thm:ehwrong}. For
$N=2$ the result is contained in \autoref{thm:dimensionsinfinitesimal}. For $N>1$ we
consider a hypersurface $M_\gamma\subset \CC{2},\,\gamma\neq 0$
and write it up near the origin as $v=\Theta_\gamma(z\bar z,u)$
(here $w=u+iv$). Set
$$M:=\left\{v=\Theta_\gamma(z_1\bar z_1+...+z_{N-1}\bar
z_{N-1},u)\right\}\subset\CC{N}$$ (here we denote by
$(z_1,...,z_{N-1},w)$ the coordinates in $\CC{N}$). Then $M$ is a
real-analytic holomorphically nondegenerate hypersurface. It
follows from the fact that
$Y_2=-\frac{2\tau^4\chi'}{\chi\tau'}z\frac{\partial}{\partial
z}+\frac{2\tau^4}{\tau'}\frac{\partial}{\partial
w}\in\mathfrak{aut}^\infty\,(M_\gamma,0)$ that the vector field
$$Y=-\frac{2\tau^4\chi'}{\chi\tau'}\left(z_1\frac{\partial}{\partial
z_1}+...+z_{N-1}\frac{\partial}{\partial
z_{N-1}}\right)+\frac{2\tau^4}{\tau'}\frac{\partial}{\partial
w}\in\mathfrak{aut}^\infty\,(M,0).$$ Then arguments identical to the
ones used for the proof of \autoref{thm:dimensionsinfinitesimal} show that
$Y\not\in\mathfrak{hol}^\omega\,(M,0)$, i.e.,
$\mathfrak{aut}^\omega\,(M,0)\subsetneq\mathfrak{aut}^\infty\,(M,0)$ and
$\mathfrak{hol}^\omega\,(M,0)\subsetneq\mathfrak{hol}^\infty\,(M,0)$. This
proves the theorem.

\end{proof}

We say that a real-analytic CR-submanifold $M\subset\CC{N}$ is \it
orbitally homogeneous, \rm if for each CR-orbit $P$ of $M$ and any
point $p\in P$ the image of $\mathfrak{hol}^\omega\,(M,0)$ under the
evaluation mapping $e_p:\,L\mapsto
L|_p,\,L\in\mathfrak{hol}\,(M,0)$ contains $T_p P$. We say that a
real-analytic CR-submanifold $M\subset\CC{N}$ is \it orbitally
CR-homogeneous, \rm  if in the above definition
$\mathfrak{hol}^\omega\,(M,0)$ is replaced by
$\mathfrak{hol}^\infty\,(M,0)$. For an orbitally homogeneous (resp.
orbitally CR-homogeneous) CR-manifold its germs at any two points,
belonging to the same CR-orbit, are holomorphically (resp.
$C^\infty$ CR) equivalent. For minimal CR-manifolds the orbital
homogeneity is equivalent to the standard local homogeneity (see
\cite{zaitsev} for details of the concept). It turns out that in
the nonminimal settings the concepts of orbital and CR-orbital
homogeneities respectively are distinct, even for the case of
holomorphically nondegenerate hypersurfaces.

\begin{theorem}
Any hypersurface $M_\gamma$ with $\gamma\neq 0$ as above is
orbitally CR-homogeneous, but not orbitally homogeneous.
\end{theorem}

\begin{proof}
Clearly, the orbit of the origin under the action of the Lie
algebra $\mathfrak{hol}^\infty\,(M_\gamma,0)$ of CR-vector fields
coincides with the 2-dimensional CR-orbit $X=\{w=0\}$, and we
obtain the orbital CR-homogeneity of $X$. The local homogeneity of
the maximal-dimensional CR-orbits
$M_\gamma^\pm=M_\gamma\cap\{\pm\re w>0\}$ follows from their
sphericity at each point. Thus $M_\gamma$ is orbitally
CR-homogeneous. The fact that $M_\gamma$ is not orbitally
homogeneous follows from \autoref{thm:dimensionsinfinitesimal}.
\end{proof}

\end{document}